\documentclass[a4paper,10pt]{article}
\usepackage[margin=1in]{geometry}
\usepackage{amsmath, amsthm, amsfonts, amssymb}
\usepackage[utf8]{inputenc}
\usepackage{mathrsfs}
\usepackage{color}

\newcommand{\dd}{\mathrm{d}}

\newcommand{\ep}{\varepsilon}

\newcommand{\din}{\dot{\,\in\,}}
\newcommand{\ddin}{\ddot{\,\in\,}}

\def\N{\mathbb{N}}
\def\R{\mathbb{R}}
\def\L{\textnormal{L}}
\def\Cc{\textnormal{C}}

\def\U{\textnormal{U}}
\def\W{\textnormal{W}}
\def\S{\textnormal{S}}
\def\E{\textnormal{E}}
\def\D{\textnormal{D}}
\def\H{\textnormal{H}}

\def\pa{\partial}
\def\ov{\overline}

\def\Id{\textnormal{Id}}

\newtheorem{Theo}{Theorem}[section]
\newtheorem{cor}[Theo]{Corollary}
\newtheorem{lem}[Theo]{Lemma}
\newtheorem{Propo}[Theo]{Proposition}
\newtheorem{definition}[Theo]{Definition}
\newtheorem{remark}[Theo]{Remark}

\begin{document}

\begin{center}
\Large{On the entropic structure of reaction-cross diffusion systems}
 \end{center}
\bigskip

\centerline{\scshape L. Desvillettes}
\medskip
{\footnotesize
  \centerline{CMLA, ENS Cachan, CNRS}
  \centerline{61 Av. du Pdt. Wilson, F-94230 Cachan Cedex, France}
\centerline{E-mail: desville@cmla.ens-cachan.fr}}
\bigskip

\centerline{\scshape Th. Lepoutre}
\medskip
{\footnotesize
\centerline{INRIA
}
\centerline{Universit\'e de Lyon}
\centerline{CNRS UMR 5208}
\centerline{Universit\'e Lyon 1}
\centerline{Institut Camille Jordan}
\centerline{43 blvd. du 11 novembre 1918}
\centerline{F-69622 Villeurbanne cedex France}
\centerline{E-mail: thomas.lepoutre@inria.fr}}
\bigskip

\centerline{\scshape A. Moussa}
\medskip
{\footnotesize
\centerline{UPMC Universit\'e Paris 06 \& CNRS}
\centerline{UMR 7598, LJLL, F-75005, Paris, France}
\centerline{E-mail : moussa@ann.jussieu.fr}}
\bigskip

\centerline{\scshape A. Trescases}
\medskip
{\footnotesize
  \centerline{CMLA, ENS Cachan, CNRS}
  \centerline{61 Av. du Pdt. Wilson, F-94230 Cachan Cedex, France}
\centerline{E-mail: trescase@cmla.ens-cachan.fr}}
\bigskip

\begin{abstract}
This paper is devoted to the study of systems of reaction-cross diffusion equations arising
in population dynamics. New results of existence of weak solutions are presented, allowing
to treat systems of two equations in which one of the cross diffusions is convex, while the other one is concave. The treatment of such cases involves a general study of the structure
of Lyapunov functionals for cross diffusion systems, and the introduction of a new scheme of approximation, which provides simplified proofs of existence.
\end{abstract}

\section{Introduction}


\subsection{The system}\label{subsec:syst}
All the systems that we are going to tackle in this study take the following form 
\begin{align}\label{eq:syst_multi}
\partial_t u_i -\Delta[a_i(u_1,\dots,u_I)\,u_i] &= r_i(u_1,\cdots,u_I)\,u_i,\text{ on }\R_+\times\Omega, \text{ for }i=1,\ldots,I,
\end{align}
where $a_i, r_i : \R^I \rightarrow \R$, the unknown being here the family of $I$ nonnegative functions $u_1,\dots,u_I$. Here and hereafter, $\Omega$ is a bounded open set of $\R^d$, whose unit normal outward vector at point $x\in\pa\Omega$ is denoted by $n=n(x)$. The system is completed with initial conditions given by a family of functions $u_1^{in},\dots,u_I^{in}$, and homogeneous Neumann boundary conditions, that is $\pa_n u_i:=\nabla u_i \cdot n =0$ on $\R_+\times\partial\Omega$ for $i=1,\dots ,I$. 
\par
Introducing the vectorial notations $U:=(u_1,\dots,u_I)$ and $U^{in}:=(u_1^{in},\dots,u_I^{in})$, and denoting by $A$ and $R$ the maps $A: U\mapsto (a_i(U)\,u_i)_i$ and $R: U\mapsto (r_i(U)\,u_i)_i$, the previous system \eqref{eq:syst_multi} can be summarized in the vectorial equations
\begin{align}
\label{eq:syst_multi_vec}\partial_t U -\Delta [A(U)] &= R(U),\text{ on }\R_+\times\Omega, \\
\label{eq:syst_multi_vec_neum}\pa_n U &= 0,\text{ on }\R_+\times\pa \Omega, \\
\label{eq:syst_multi_vec_init} U(0) &=U^{in}, \text{ on }\Omega.
\end{align}
\medskip

Such systems have received a lot of attention lately (cf. \cite{Chen2006}, \cite{deslepmou}, for example). Their origin is to be found in the seminal paper \cite{Shigesada1979}, where one typical example 
(now known in the literature as the SKT system) is introduced: namely, $I=2$, and $A,R$ are given by affine functions:
$$ a_1(u_1,u_2) = d_1 + d_{11}\, u_1 + d_{12}\, u_2, 
\qquad a_2(u_1,u_2) = d_2 + d_{21}\, u_1 + d_{22}\, u_2, $$
$$ r_1(u_1,u_2) = R_1 - R_{11}\, u_1 - R_{12}\, u_2, 
\qquad r_2(u_1,u_2) = R_2 - R_{21}\, u_1 - R_{22}\, u_2 . $$
The corresponding equations model the evolution of individuals belonging to two
species in competition, which increase their diffusion rate in order to
avoid the individuals of the other (or the same) species.  This evolution can lead 
 to the formation of patterns when $t\to \infty$ (cf. \cite{Shigesada1979}).
There is an important literature on the question of existence of global classical solutions to the SKT systems (but only for particular cases). To summarize the approach, one can prove local existence of classical solutions 
 using Amman's theorem \cite{amann90b} and the difficulty relies then in the proof of bounds on the solutions in suitable Sobolev spaces to prevent blow up. The works in this direction always have restrictions on the coefficients (typically no cross diffusion is introduced for one of the species, cf. for example \cite{DesTres}) and/or on dimension. In particular, it is worth noticing that existence of global classical solutions to the full SKT system (that is, when all coefficients are positive) remains a challenging open problem except in dimension 1. 
\medskip

Our work however deals with the existence of global weak solutions for which an important step forward was made by Chen and J\"ungel in
\cite{Chen2006}. They indeed showed that a hidden Lyapunov-style 
functional  (that is, a Lyapunov functional if the terms $r_1,r_2$
are neglected) exists for this system 
without restriction on coefficients such as strong self diffusion for instance.
In \cite{deslepmou}, this structure was shown to be robust enough to treat
functions $a_1$, $a_2$ such as
\begin{equation}\label{sta}
 a_1(u_1,u_2) = d_1 + d_{11}\, u_1^{\delta_{11}} + d_{12}\, u_2^{\delta_{12}}, 
\qquad a_2(u_1,u_2) = d_2 + d_{21}\, u_1^{\delta_{21}} + d_{22}\, u_2^{\delta_{22}}, 
\end{equation}
when $\delta_{12} \in ]0,1[$ and $\delta_{21} \in ]0,1[$.
A tool coming out of the reaction-diffusion theory (namely, duality lemmas, cf. \cite{PiSc}) 
was also introduced in the context of cross diffusion type systems in \cite{PierreD} and \cite{Lepoutre_JMPA}. It was then used associated with the entropy structure in \cite{deslepmou}
in order to extend the range of entropies that could be treated.
\medskip

In this paper, we investigate more deeply the structure of systems like 
(\ref{eq:syst_multi}), in order to exhibit as much as possible 
the conditions which enable the existence of Lyapunov-style 
functionals. Subsection \ref{subsec:genent} is directly devoted to this study.
As an application, we show that cross-diffusion terms like (\ref{sta})
can be treated as soon as the product $\delta_{12} \,\delta_{21}$ belongs to $]0,1[$, thus significantly enlarging the conditions described in \cite{deslepmou}.
\medskip

Another difficulty appearing in many works on equations involving 
cross diffusion is the difficulty, once {\it{a priori}} estimates
have been established, to write down an approximation scheme which
implies existence. 
An important issue comes from the fact that the entropy structure and the duality estimates 
are of very different nature. Therefore it is difficult to build an approximation that preserves both properties. In \cite{Chen2006} and related works, the entropy induces a better integrability, or even boundedness in \cite{jungbound,Zamponi}, making the use of duality estimates unnecessary. In \cite{deslepmou}, the approximation procedure is inspired from \cite{Chen2006} and duality is proved to be satisfied  \emph{a posteriori} (in fact at some intermediate step).
This is related to a lack of robustness of the Lyapunov-style 
functionals (it is indeed difficult to extend them to the approximate system). We therefore present in this paper a new approximation scheme
which is definitely easier to grasp than those presented in
\cite{Chen2006} or \cite{deslepmou}. Indeed, we use a time-discretized version of the equation for which existence can be obtained thanks to a standard (Schauder-type) fixed-point theorem, and which conserves the structure of the time-continuous equation from the viewpoint of {\it{a priori}} estimates. As a consequence, the passage to the limit when the discretization step goes to $0$ is not much more difficult than the passage to the limit in a sequence of solutions to the equation (that is, the weak stability of the equation). Note that a related (yet different) time-discrete approximation was recently used (together with a regularization) in the context of cross-diffusion systems in \cite{jungbound}.

\subsection{Assumptions}

Throughout this study, we will state and prove results necessitating one or several of the following assumptions on the parameters of (\ref{eq:syst_multi}) :
\begin{itemize}
\item[\textbf{H1}] The functions $a_i$ and $r_i$ are continuous from $\R_+^I$ to $\R$.
\item[\textbf{H2}] For all $i$, $a_i$ is lower bounded by some positive constant $\alpha>0$, and $r_i$ is upper bounded by a positive constant $\rho>0$. That is $a_i(U) \ge \alpha > 0$, 
$r_i(U)\le \rho$ for all $U\in \R_+^I$.
\item[\textbf{H3}] $A$ is a homeomorphism from $\R_+^I$ to itself. 
\end{itemize}
The set of all these assumptions will be invoked at the beginning of each statement (if needed) by writing $(\textbf{H})$. Assumption \textbf{H3} will be of utmost importance during the establishment of our approximation scheme. Assumptions \textbf{H1} and \textbf{H2} are usually
 easily checked and appear quite natural. It is however less clear to understand what type of systems satisfies \textbf{H3}. We therefore describe in Section \ref{subsec:satisfH3} some examples of systems satisfying \textbf{H3}.

\subsection{Notations} 
Since we will always work on $Q_T:=[0,T]\times\Omega$, we will simply denote by $\L^p_t(\L^q_{x})$ the corresponding evolution spaces and we will use the same convention for Sobolev spaces $\H^\ell$ and $\W^{\ell,p}$ (for $\ell\in\N, p\in[1,\infty]$). Cones of nonnegative functions will be noted specified with a $+$ subscript, for instance $\L^\infty(\Omega)_+$ is the cone of all essentially bounded nonnegative functions.   For $p\in[1,\infty[$,
 we will sometimes use the notation $\L^{p^+}$ to speak of the set of all $\L^q$ functions with $q>p$. We define also the space $\H^1_m(\Omega)$ of $\H^1(\Omega)$ functions having zero mean. Its dual is denoted by $\H^{-1}_m(\Omega)$. For any space of functions defined on $\Omega$ whose gradient has a well-defined trace on $\partial\Omega$ (such as $\H^2(\Omega)$ or $\mathscr{C}^\infty(\overline{\Omega})$ for instance), we add the subscript $\nu$ (the former spaces becoming then $\H^2_\nu(\Omega)$ and $\mathscr{C}^\infty_\nu(\overline{\Omega})$) when we wish to consider the subspaces of functions satisfying the homogeneous Neumann boundary condition.
 
\vspace{2mm}

 Given a normed space $X$, we will always denote by $\|\cdot\|_X$ its norm, except for $\L^p$ spaces for which we often write $\|\cdot\|_p$. If $(x_n)_n$ is a sequence of $X$, $(x_n)_n\din X$, and $(x_n)_n\ddin X$ respectively mean that $(x_n)_n$ is bounded in $X$, and $(x_n)_n$ is relatively compact in $X$.


\vspace{2mm}

The symbol $|\cdot|$ will always represent the Euclidian norm (but possibly in different dimensions depending of the context), whereas $|\cdot|_1$ will be the Manhattan norm. Finally, given two vectors $X=(x_i)_i$ and $Y=(y_i)_i$ of $\R^I$, we write $X\leq Y$ whenever $x_i \leq y_i$ for all $i$. We extend this order relation to $\R^I$ valued functions $f(X)$ and write, for a real number $c$, $f(X) \leq c$ when $f(X) \leq C=(c,\dots,c)$. The same convention is used
 for $<$.

\subsection{Main application}\label{maina}

The main application of the methods developed in this work is
a new theorem of existence of (very) weak solutions for systems with the same structure as in \cite{deslepmou}:

\begin{align}
\label{eq:u:reac}\partial_t u_1 -\Delta \Big[ u_1\,(d_1 + u_2^{\gamma_2}) \Big] = u_1 \,(\rho_1 -u_1^{s_{11}}-u_2^{s_{12}}),& \quad\text{ on }\R_+\times\Omega \\
\label{eq:v:reac}\partial_t u_2 -\Delta \Big[ u_2\,(d_2 + u_1^{\gamma_1}) \Big] = u_2 \, (\rho_2 -u_2^{s_{22}}-u_1^{s_{21}}),& \quad\text{ on }\R_+\times\Omega \\
\label{eq:uv_neum} \pa_n u=\pa_n v = 0,& \quad\text{ on }\R_+\times\pa \Omega.
\end{align}

We introduce the
\begin{definition}[(Very) Weak solution]\label{def:weak_sol} We consider $d_1,d_2>0$, $\rho_1, \rho_2>0$, $\gamma_1, \gamma_2>0$, and $s_{ij}>0$ ($i,j = 1,2$). Let $u_1^{in}$, $u_2^{in}$ be two nonnegative functions in $\L^1(\Omega)$. For $u_1$, $u_2$ two nonnegative functions in $\L^1_\textnormal{loc}(\R_+,\L^1(\Omega))$, we say that $(u_1,u_2)$ is a (very) weak solution of \eqref{eq:u:reac}--\eqref{eq:uv_neum} with initial conditions $(u_1^{in},u_2^{in})$ if $(u_1,u_2)$ satisfies
\begin{eqnarray}\label{def:weak_wellpose}
u_1\,[ u_1^{s_{11}} + u_2^{s_{12}} + u_2^{\gamma_{2}}]\,+\,u_2\,[ u_1^{s_{21}} + u_2^{s_{22}} + u_1^{\gamma_{1}}] \, \in\, \L^1_\textnormal{loc}(\R_+,\L^1(\Omega)),
\end{eqnarray}
and for any $\psi_1,\,\psi_2 \in \mathscr{C}^1_c(\R_+;\mathscr{C}^2_{\nu}(\overline{\Omega}))$,
\begin{equation}\begin{split}\label{def:weak_formul1}
 & - \int_{\Omega} u_1^{in}(x)\, \psi_1(0,x)\, \dd x - \int_0^{\infty}\int_{\Omega} u_1(t,x)  \,\partial_t \psi_1(t,x)\, \dd x\,\dd t \\
 & - \int_0^{\infty}\int_{\Omega} \Delta \psi_1(t,x)\,  \Big[ d_1 +  u_2(t,x)^{\gamma_2}\Big]\, u_1(t,x)\, \dd x\,\dd t \\
= &\int_0^{\infty}\int_{\Omega} \psi_1(t,x)\,u_1(t,x)\,\big(\rho_1-u_1(t,x)^{s_{11}}-u_2(t,x)^{s_{12}}\big) \, \dd x\,\dd t,
\end{split}\end{equation}
and
\begin{equation}\begin{split}\label{def:weak_formul2}
 & - \int_{\Omega} u_2^{in}(x)\, \psi_2(0,x)\, \dd x - \int_0^{\infty}\int_{\Omega} u_2(t,x) \,\partial_t \psi_2(t,x)\, \dd x\,\dd t \\
 & - \int_0^{\infty}\int_{\Omega} \Delta \psi_2(t,x)\,  \Big[ d_2 +  u_1(t,x)^{\gamma_1}\Big]\, u_2(t,x)\, \dd x\,\dd t \\
= &\int_0^{\infty}\int_{\Omega} \psi_2(t,x)\,u_2(t,x)\,\big(\rho_2-u_2(t,x)^{s_{22}}-u_1(t,x)^{s_{21}}\big) \, \dd x\,\dd t.
\end{split}\end{equation}
\end{definition}
Our theorem writes
\begin{Theo}\label{th:globalex}
Let $\Omega$ be a smooth ($\mathscr{C}^2$) bounded open subset of $\R^d$ ($d\ge 1$). Consider $\gamma_2>1$, $0<\gamma_1<1/\gamma_2$ and $\rho_i>0$, $s_{ij}>0$, $d_i>0$, $i,j=1,2$, with $s_{11}<1$, $s_{12}<\gamma_2+s_{22}/2$ and $s_{21}<2$. Let $U^{in}:=(u_1^{in},u_2^{in})\in (\L^1\cap\H^{-1}_m)(\Omega)\times (L^{\gamma_2}\cap\H^{-1}_m)(\Omega)$ be a couple of nonnegative initial data.
\par 
Then, there exists a couple $U:=(u_1,u_2)$ of nonnegative functions which is a (very) weak solution
to \eqref{eq:u:reac}--\eqref{eq:uv_neum} in the sense of Definition \ref{def:weak_sol}, and satisfies, for all $s>0$,
\begin{equation}\label{eq:estimationL2}
 \int_0^s \int_{\Omega} (u_1+u_2)\,(u_1^{\gamma_1} u_2 +u_2^{\gamma_2} u_1+ u_1+u_2) \, \dd x\,\dd t \le D_s, 
\end{equation}
\begin{equation} \label{24n}
 \sup_{t\in [0,s]} ||u_i(t, \cdot)||_{\L^1(\Omega)} \le e^{\rho_i s} ||u_i^{in}||_{\L^1(\Omega)},
 \end{equation}
\begin{equation} \label{nne}\begin{split}
 \sup_{t\in [0,s]} \int_\Omega u_2(t,\cdot)^{\gamma_2} + \int_0^s \int_{\Omega} u_2^{\gamma_2} \,\left( u_1^{s_{21}} +u_2^{s_{22}}  \right) \, \dd x \dd t &\\
 + \int_0^s \int_{\Omega} \Big\{|\nabla u_1^{\gamma_1/2}|^2  + |\nabla u_2^{\gamma_2/2}|^2 + \left|\nabla \sqrt{u_1^{\gamma_1} \,u_2^{\gamma_2}}\right|^2 \Big\} \, \dd x \dd t &\\
 \leq K_s \, ( 1+& \|u_1^{in}\|_{\L^1(\Omega)}+ \|u_2^{in}\|^{\gamma_2}_{\L^{\gamma_2}(\Omega)}).
\end{split}\end{equation}
The positive constants $K_s$ and $D_s$ used above only depend on $s$, $\Omega$ and the data of the equations ($\rho_i$, $d_i$, $\gamma_i$, $s_{ij}$). The constant $D_s$ in (\ref{eq:estimationL2}) also depends on $\|U^{in}\|_{\H^{-1}_m(\Omega)^2}$. 
Both functions $s\mapsto D_s$ and $s\mapsto K_s$ may be chosen continuous and belong in particular to $\L^\infty_\textnormal{loc}(\R_+)$.
\end{Theo}

\begin{remark}
We consider in this theorem 
 the case $\gamma_2>1$, $\gamma_1<1$, which falls outside of  the scope of the systems studied in \cite{deslepmou}. 
\medskip

Note that Theorem \ref{th:globalex}  and its proof still hold when one  adds some positive constants in front of the non-linearities.
\medskip

A more technical aspect concerns the self-diffusion, that we have chosen to disregard here.
 As it was the case in \cite{deslepmou}, adding  self-diffusion terms
tends in fact 
here 
 to facilitate the study of the system, and  it 
 gives rise to extra estimates on the  gradients of the densities. 
\medskip

Note also that Theorem \ref{th:globalex} may be generalized to the case when power rate diffusion coefficients are replaced by mere functions $a_{ij}(u_i)$, with \emph{ad hoc} assumptions of regularity/increasingness/concavity on the functions $a_{ij}$. An extension to reaction coefficients $R$ different from power laws is also certainly possible.
\medskip

The condition $s_{12}<\gamma_2+s_{22}/2$ is in fact not optimal. It can be improved using different interpolations. As it will be seen, we have $(A)$ $\int_0^T\int_{\Omega} u_1^{\max(2,s_{21})} u_2^{\gamma_2}<\infty$, $(B)$ $\int_0^T\int_{\Omega} u_2^{\gamma_2+s_{22}}<\infty$ and $(C)$ $\int_0^T\int_{\Omega} u_1^{\gamma_1} u_2^{2}<\infty$. Interpolating between $(A)$ and $(B)$ leads to the condition $s_{12}<\gamma_2+s_{22}(1-1/\max(2,s_{21}))$ while interpolating between $(A)$ and $(C)$ leads to the condition $s_{12}<2-(2-\gamma_2)(1-\gamma_1)/(\max(2,s_{21})-\gamma_1)$. All in all, we get the sufficient condition
$$s_{12}<\max \bigg(\gamma_2+s_{22}/2,\gamma_2+s_{22}(1-1/s_{21}),2-(2-\gamma_2)(1-\gamma_1)/2-\gamma_1),$$
$$ 2-(2-\gamma_2)(1-\gamma_1)/(s_{21}-\gamma_1) \, \bigg).$$
In this formula, the four different expressions can lead to the best condition on $s_{12}$, depending on the coefficients $s_{21}, s_{22}, \gamma_1, \gamma_2$. For instance : if $s_{21}=1, s_{22}=2, \gamma_2=2$, the best condition is given by the first expression; if $s_{21}=4, s_{22}=2, \gamma_2=2$, the best condition is given by the second expression, if $s_{21}=1, s_{22}=1/5, \gamma_2=3/2$, the best condition is given by the third expression, if $s_{21}=4, s_{22}=1/5, \gamma_2=3/2$, the best condition is given by the fourth expression.
\medskip

Thanks to estimates \eqref{eq:estimationL2} and \eqref{nne}, we can show that the quantity $\nabla \left\{[d_1 + u_2^{\gamma_2}]u_1\right\}$ lies in the space $\L^1_{\textnormal{loc}}(\R_+,\L^1(\Omega))$, so that $u_1$ is actually a weak solution of \eqref{eq:u:reac} (in the sense that in the weak formulation \eqref{def:weak_formul1} we can replace $-\iint \Delta \psi_1 [d_1 + u_2^{\gamma_2}]u_1 $ by $ \iint \nabla \psi_1 \cdot\nabla \left\{[d_1 + u_2^{\gamma_2}]u_1\right\}$ for all $\psi_1 \in \mathscr{C}^1_c(\R_+;\mathscr{C}^2_{\nu}(\overline{\Omega}))$, and therefore by a density argument enlarge the set of test functions $\psi_1$ to $\mathscr{C}^1_c(\R_+;\mathscr{C}^1_{\nu}(\overline{\Omega}))$). If $\gamma_2\le 2$, we can use estimates \eqref{eq:estimationL2} and \eqref{nne} to show that $\nabla \left\{[d_2 + u_1^{\gamma_1}]u_2\right\}$ also lies in $\L^1_{\textnormal{loc}}(\R_+,\L^1(\Omega))$, so that $u_2$ is actually a weak solution of \eqref{eq:v:reac} (in an analogous sense).
We use for that the computation $\nabla u_2 = 2\gamma_2^{-1} u_2^{1-\gamma_2/2}\nabla u_2^{\gamma_2/2}\in \L^2_{\textnormal{loc}}(\R_+,\L^2(\Omega))\times \L^2_{\textnormal{loc}}(\R_+,\L^2(\Omega))$. However, in the case when $\gamma_2>2$, this computation does not hold true anymore,  and in this case our theorem only gives very weak solutions.
\end{remark}

\subsection{Structure of the paper}

We begin by introducing and studying a general (for $I$ species with $I\ge1$) semi-discrete scheme in Section \ref{sec:semi_discr_scheme}. More precisely, we prove the existence of a solution for the discretized system and we show estimates satisfied by the solution (some of them are uniform in the time step and some are not). Section \ref{sec:twospec} is devoted to the study of the hidden entropy structure for a class of two-species cross-diffusion systems including \eqref{eq:u:reac}--\eqref{eq:uv_neum}. We prove Theorem \ref{th:globalex} in Section \ref{sec:globweak}. Finally, we exhibit some examples of systems satisfying Assumption \textbf{H3} (which allows to use the results of Section 2) and recall some elliptic estimates in the Appendix (Section 5). Sections \ref{sec:semi_discr_scheme} and \ref{sec:twospec} are independent, whereas the proof of Theorem \ref{th:globalex} in Section \ref{sec:globweak} uses the results of Section \ref{sec:semi_discr_scheme} and relies on the entropy structure detailed in Section \ref{sec:twospec}.

\section{Semi-discrete scheme}\label{sec:semi_discr_scheme}

We begin here the presentation of general statements which will be used in the proof of Theorem \ref{th:globalex}. More precisely, we intend in this section to introduce a semi-discrete implicit scheme to approximate multi-dimensional systems of the form \eqref{eq:syst_multi}. Although many breakthroughs occurred in the mathematical understanding of cross-diffusion systems in the recent years (see \cite{Chen2006,deslepmou} and the references therein), the approximation procedure of these systems frequently leads to intricate or technical methods. The main reason is that the hidden entropy structure often relies on functionals defined on nonlinear subspaces of $\R^I$ ($\R_+^I$ for instance), and a condition as simple as ``$u_i$ is nonnegative'' is not easily kept during the approximation process.  

\vspace{2mm}

The scheme is based on the following semi-discretization ($1\leq k\leq N-1$, $N=T/\tau$, $T>0$)
\begin{equation}\label{eq:scheme}\begin{split}
\frac{U^k-U^{k-1}}{\tau} -\Delta [A(U^k)] &= R(U^k),\text{ on }\Omega,\\
\pa_n A(U^k) &= 0,\text{ on }\pa\Omega.
\end{split}\end{equation}
We introduce the
\begin{definition}[Strong solution]\label{def:strongsol}
\emph{(\textbf{H})}\textbf{.}
 Let $\tau>0$ and $U^{k-1}\in \L^\infty(\Omega)_+^I$. We say that a nonnegative vector-valued function $U^k$ is a strong solution of \eqref{eq:scheme} if $U^k$ lies in $\L^\infty(\Omega)^I$, $A(U^k)$ lies in $\H^2_{\nu}(\Omega)^I$ and the first equation in \eqref{eq:scheme} is satisfied almost everywhere on $\Omega$.
\end{definition}


\medskip

Our results concerning this scheme are summarized in the
\begin{Theo}[\textbf{H}]\label{th:existence_scheme_it}
Let $\Omega$ be a bounded open set of $\R^d$ with smooth boundary.  Fix $T>0$ and an integer $N$ large enough such that $\rho\tau<1/2$, where $\tau:=T/N$. Fix $\eta>0$ and a vector-valued function  $\L^\infty(\Omega)^I\ni U^0\geq \eta$. Then there exists a sequence of positive vector-valued functions $(U^k)_{1\leq k\leq N-1}$ in $\L^\infty(\Omega)^I$ which solve \eqref{eq:scheme} (in the sense of Definition \ref{def:strongsol}). Furthermore, it satisfies the following estimates :
for all $k\geq 1$ and  $p\in[1,\infty[$,
\begin{align}\label{es:depend_on_tau1}
U^k &\in\mathscr{C}^0(\overline{\Omega})^I, \\
U^k &\geq \eta_{A,R,\tau}  \text{ on }\overline{\Omega},\\
\label{es:depend_on_tau3} A(U^k) &\in \W^{2,p}_\nu(\Omega)^I,
\end{align}
where $\eta_{A,R,\tau}>0$ is a positive constant depending on the maps $A$ and $R$ and $\tau$, and
\begin{align}
\label{ineq:l1} \max_{0\leq k\leq N-1}\int_\Omega U^k  &\leq  2^{2\rho \tau N} \int_\Omega U^0, \\
\label{ineq:l1bis} \sum_{k=1}^{N-1} \tau \int_\Omega \left(\rho U^k - R(U^k)\right) &\leq 2^{2\rho \tau N} \int_\Omega U^0,\\
\label{for38}
\sum_{k=0}^{N-1} \tau \int_\Omega \left(\sum_{i=1}^I u_i^k\right)\left(\sum_{i=1}^I a_i(U^k) u_i^k\right) &\leq  C(\Omega,U^0,A, \rho, N\tau),
\end{align}
where $C(\Omega,U^0,A, \rho, N\tau)$ is a positive constant depending only on $\Omega$, $A$, $\rho$, $N\tau$ and $\|U^0\|_{\L^1\cap\H^{-1}_m(\Omega)}$.
\end{Theo}

\begin{remark}\label{rem:deptau}
Estimates \eqref{es:depend_on_tau1}--\eqref{es:depend_on_tau3} strongly depend on $\tau$. In particular, they will be lost when we pass to the limit $\tau\rightarrow 0$ during the proof of existence of global solutions in Section~\ref{sec:globweak}.
These estimates are however crucial in order to perform rigorous computations and obtain uniform estimates on the scheme.

\medskip

Consider $T>0$ as fixed. Estimates \eqref{ineq:l1} and \eqref{ineq:l1bis} do not depend on $\tau$, since $\tau N=T$. Estimate \eqref{for38} is in fact also uniform w.r.t. $\tau$,  in the sense that it yields a limiting estimate in 
the limit $\tau \rightarrow 0$.

\medskip

The main feature of the scheme \eqref{eq:scheme} is its ability to preserve the entropy structure described in Section \ref{sec:twospec}: indeed, when such a structure exists at the continuous level, we have good hope to get similar estimates (in particular, the gradient estimates described in Remark \ref{rem:gradients}) at the semi-discrete level, uniformly in $\tau$.
 Since formalizing this property in an abstract theorem requires very
 specific assumptions, this will only be shown on a specific example in Section \ref{sec:globweak}: see Proposition \ref{propo:discrent}.

\end{remark}

The proof of the existence of the family $(U^k)_k$ solving \eqref{eq:scheme} is done in Subsection \ref{sec:exsch}. The proof of the various estimates is done in Subsection \ref{sec:apriori}.

\subsection{Existence theory for the scheme}\label{sec:exsch}

 We plan in this section to build step by step a family $(U^k)_{1\leq k\leq N-1}$ solution of \eqref{eq:scheme} (for a given $U^0$, bounded and nonnegative). For simplicity, we drop the subscripts and rewrite the scheme \eqref{eq:scheme} with the notations $U:=U^k$ and $S:=U^{k-1}$:
\begin{equation}\label{eq:scheme_k_droped}\begin{split}
U -\tau\Delta [A(U)] &= S +\tau R(U) \hspace{1cm} \text{on } \Omega,\\
\pa_n A(U) &=0 \hspace{2.78cm} \text{on } \pa\Omega.
\end{split}\end{equation}
The existence of the family $(U^k)_{1\leq k\leq N-1}$ is a consequence of the iterated use of the following Theorem
\begin{Theo}[\textbf{H}]\label{th:existence_scheme} Let $\Omega$ be a bounded open set of $\R^d$ with smooth boundary. If $S\in\L^\infty(\Omega)^I_+$, then for all $\tau>0$ such that $\rho \tau< 1/2$, there exists  $U\in \L^\infty(\Omega)^I_+$ which is a strong solution of \eqref{eq:scheme_k_droped} (in the sense of Definition \ref{def:strongsol}).
 Furthermore, this solution satisfies (for some $C(\Omega,I \|S\|_\infty)$ only depending on
$\Omega$ and $I \|S\|_\infty$) the estimate
\begin{equation*}
\|U\|_{\infty} \le \frac{1}{\alpha \tau} C(\Omega,I \|S\|_\infty).
\end{equation*}
\end{Theo}
The proof of Theorem \ref{th:existence_scheme} is based on a fixed point method that we present in Subsection \ref{s41}.

\subsubsection{Fixed point}
{\sl{Proof of Theorem \ref{th:existence_scheme}}}\label{s41}: Our aim is to apply the Leray-Schauder fixed point Theorem, which can be stated in the following way (see for example Theorem 11.6 p.286 in \cite{gilbarg}):

\begin{Theo}\label{th:Leray-Schauder} [Leray-Schauder]
Let $\mathfrak{B}$ be a Banach space and let $\Lambda$ be a continuous and compact mapping of $[0,1]\times\mathfrak{B}$ into $\mathfrak{B}$ such that $\Lambda(0,x)=0$ for all $x\in\mathfrak{B}$. Suppose that there exists a constant $L>0$ such that
 for all $\sigma\in [0,1]$,
$$\Lambda(\sigma,x)=x \implies \|x\|_\mathfrak{B}<L.$$
Then the mapping $\Lambda(1,\cdot)$ of $\mathfrak{B}$ into itself has at least one fixed point.
\end{Theo}

In order to apply Theorem \ref{th:Leray-Schauder} to our problem, we start by defining, for $U\in \L^\infty(\Omega)^I$,
\begin{align*}
&\overline{M}(U) :=  \max\left\{\frac{M_{p,\Omega}}{2\rho},\frac{2 +\tau  \operatorname*{\max}_{i=1 \cdots n} \| r_i(U) \|_\infty}{\alpha}\right\},
\end{align*}
where $p>d/2$ is fixed and $M_{p,\Omega}$ is the constant defined in Lemma \ref{lem:ell}. The quantity $\overline{M}(U)$ is well defined because $r_i$ is assumed to be continuous (so that  $r_i(U) \in\L^\infty(\Omega)$). The definition of $\overline{M}(U)$ and the fact that $\rho \tau < 1/2$ imply the two following inequalities (almost everywhere on $\Omega$)
\begin{align*}
\overline{M}(U) &> \tau M_{p,\Omega}, \\
\overline{M}(U)A(U) -U  +\tau R(U)&\geq  0,
\end{align*}
where the second (vectorial) inequality has to be understood coordinates by coordinates.

Consider now the following maps (here both $A$ and $R$ are extended
 to continuous functions on $\R^I$ by parity):
\begin{align*}
\Psi: \L^\infty(\Omega)^I &\longrightarrow \L^\infty(\Omega)^I_+ \times (\tau M_{p,\Omega}, + \infty) \\
U &\longmapsto ( S + \overline{M}(U)A(U)-U+\tau R(U) , \overline{M}(U) ),\\
&\\
\Theta : \L^\infty(\Omega)_+^I \times (\tau M_{p,\Omega}, + \infty) &\longrightarrow \L^\infty(\Omega)^I_+ \\
(U,M) &\longmapsto  (M \Id-\tau\Delta)^{-1} U, \\
&\\
\Phi: \L^\infty(\Omega)_+^I &\longrightarrow \L^\infty(\Omega)_+^I \\
 U &\longmapsto A^{-1}(U),
\end{align*}
where the inverse operator in the definition of $\Theta$ has to be understood with homogeneous Neumann boundary conditions on $\partial\Omega$ and in the strong sense (that is, $\Theta(U,M)\in \H^2_{\nu}(\Omega)^I$ and $(M \Id-\tau\Delta)(\Theta(U,M))$ $=U$ a.e. on $\Omega$). 
\medskip

Notice that  for all $M$, $\Theta(\cdot, M)$ indeed sends $\L^\infty(\Omega)^I$ to $\L^\infty(\Omega)^I$ thanks to Lemma \ref{lem:ell} of the Appendix, since $p>d/2$, and 
remembering Sobolev embedding $\W^{2,p}(\Omega)\hookrightarrow\L^\infty(\Omega)$. Using the maximum principle, we also can see that $\Theta(\cdot, M)$ preserves the nonnegativeness of the components, so that $\Theta(\L^\infty(\Omega)_+^I \times (\tau M_{p,\Omega}, + \infty)) \subset\L^\infty(\Omega)_+^I$.
\par
 We can therefore consider the mapping  $\Phi\circ\Theta \circ \Psi$,
 and it is clear that any fixed point of this mapping will give us a solution of the discretized system \eqref{eq:scheme_k_droped}.  We hence plan to apply the Leray-Schauder  Theorem to prove the existence of such a fixed point. We consider for this purpose the map $\Lambda(\sigma,\cdot) := \Phi \circ \sigma \Theta \circ \Psi$. We obviously have $\Lambda(0,\cdot) = 0$. Let us first check the continuity and compactness of $\Lambda$, and then look for a uniform estimate for the fixed points of the applications $\Lambda(\sigma,\cdot)$, to prove that $\Lambda(1,\cdot)=\Phi\circ\Theta \circ \Psi$ indeed has a fixed point.

\subsubsection{Compactness and continuity of $\Lambda$.}

\begin{lem}\emph{(\textbf{H})}\label{le:S_compact}
The map $\Lambda : [0,1]\times \L^\infty(\Omega)^I\rightarrow \L^\infty(\Omega)^I$ is compact and continuous.
\end{lem}
\begin{proof}
Thanks to the Heine-Cantor Theorem and the continuity of $A$, $R$ and $A^{-1}$ (see assumptions \textbf{H}), we see that both $\Phi$ and $\Psi$ are continuous. It is hence sufficient to prove the continuity and compactness of $\sigma \Theta$ from $[0,1]\times\L^\infty(\Omega)^I_+\times (\tau M_{p,\Omega}, + \infty)$ to $\L^\infty(\Omega)_+^I$. The compactness of this mapping is a straightforward consequence of Lemma \ref{lem:ell} of the Appendix
 together with the corresponding Sobolev embeddings. 
\par
For the continuity,
let us define  $\tilde U:=\Theta(U,M)$. By maximum principle, we see that 
$$
\|\tilde U\|_\infty \leq \frac{\|U\|_\infty}{M}\leq \frac{\|U\|_\infty}{\tau M_{p,\Omega}}.
$$
Furthermore, for given $(U,M),(U',M')$, defining similarly $\tilde U':=\Theta(U',M')$, we can write
$$
M(\tilde U-\tilde U')-\tau\Delta (\tilde U-\tilde U')=(U-U')+(M'-M)\tilde U',\quad \partial_n(\tilde U-\tilde U')=0.
$$
Still by maximum principle, we immediately get
\begin{align*}
\|\tilde U-\tilde U'\|_\infty&\leq \frac{1}{M}\left(\|U-U'\|_\infty +|M'-M|\|\tilde U'\|_\infty\right)\\
&\leq \frac{1}{\tau M_{p,\Omega}}\left(\|U-U'\|_\infty +|M'-M|\frac{\|U'\|_\infty}{\tau M_{p,\Omega}}\right),
\end{align*}
which yields the continuity of the application $\Theta$, and thereby of the application $\Lambda$.
\end{proof} 

\subsubsection{Estimates on fixed points of $\Lambda(\sigma,\cdot)$.}

In order to apply the Leray-Schauder fixed-point Theorem, we need an \emph{a priori} estimate (uniform in $\sigma$) on the fixed points of $\Lambda(\sigma,\cdot)$.  The fixed point equation is rewritten as
\begin{align*}
\overline{M}(U)A(U)-\tau\Delta [A(U)]=\sigma (S + \overline{M}(U)A(U)-U+\tau R(U)).
\end{align*}

We prove the 
\begin{lem}\emph{(\textbf{H})}\label{lem:estil1}
For any $\sigma \in [0,1]$, any fixed point $U\in \L^\infty(\Omega)^I$ of $\Lambda(\sigma, .)$ satisfies
\begin{align*}
\int_\Omega  |A(U)|_1 \leq C(\Omega,I \|S\|_\infty).
\end{align*}
\end{lem}
\begin{proof}
Notice first that such fixed points are necessarily nonnegative, so that $R(U) \leq \rho U$ (assumption \textbf{H2}) and
\begin{align*}
|A(U)|_1 = \sum_{i=1}^I a_i(U)u_i \in \H^2_{\nu}(\Omega).
\end{align*}
The equations associated to the fixed point imply the following vectorial inequality: 
\begin{align*}
\overline{M}(U) (1-\sigma) A(U) + \sigma (1- \tau \rho) U -\tau \Delta [A(U)]\leq \sigma S, \quad \text{ a.e. on }\Omega.
\end{align*}
Summing up each coordinate of this inequality, we get 
\begin{align}
\label{ineq:fixpoint}\overline{M}(U) (1-\sigma) |A(U)|_1 + \sigma (1- \tau \rho) |U|_1 -\tau \Delta |A(U)|_1\leq \sigma |S|_1, \quad \text{ a.e. on }\Omega,
\end{align}
which, after multiplication by $|A(U)|_1$ and integration on $\Omega$, leads to
\begin{align*}
(1-\sigma)\int_\Omega \bar{M}(U)|A(U)|_1^2 +\sigma (1- \rho \tau) \int_\Omega |A(U)|_1 |U|_1  &+\tau\int_\Omega |\nabla |A(U)|_1 |^2 \\
&\le \sigma\int_\Omega |S|_1 |A(U)|_1\,.
\end{align*}
For $\sigma=0$, we have $U=0=A(U)$, otherwise we get 
\begin{align*}
(1- \rho \tau) \int_\Omega |A(U)|_1 |U|_1 &\leq \int_\Omega |S|_1 |A(U)|_1   \leq I\|S\|_\infty \int_\Omega |A(U)|_1,
\end{align*}
whence, since $\rho\tau <1/2$, 
\begin{align}
\label{ineq:cut}\int_\Omega |A(U)|_1 |U|_1  \leq 2I\|S\|_\infty \int_\Omega |A(U)|_1.
\end{align}
Thanks to the continuity of $A$, we see that for all $R>0$, there exists $C_1(R)>0$
such that
\begin{align*}
|U|_1 \leq R \Longrightarrow |A(U)|_1 \leq C_1(R),
\end{align*}
so that small values of $U$ will be handled. On the other hand, we have 
\begin{align*}
\int |A(U)|_1 |U|_1 &\geq \int_{|U|_1 \geq R} |A(U)|_1 |U|_1 \geq R\int_{|U|_1 \geq R} |A(U)|_1.
\end{align*}
Now cutting the r.h.s. of \eqref{ineq:cut} in two terms, corresponding to  small and large values of $U$, we get
\begin{align}\label{comp:cut_integral}
\int_\Omega |A(U)|_1|U|_1 &\leq 2I\|S\|_\infty \left(\int_{|U|_1 \geq R} |A(U)|_1 + \int_{|U|_1 < R} |A(U)|_1  \right)\\
&\leq \frac{2I\|S\|_\infty}{R} \int_\Omega |A(U)|_1 |U|_1 + 2 I \|S\|_\infty |\Omega| C_1(R).
\end{align} 
Taking $R=4 I\|S\|_\infty$, we have
\begin{align*}
\int_\Omega |A(U)|_1|U|_1 \leq 4 I \|S\|_\infty |\Omega| C_1(R).
\end{align*} 
Reusing computation \eqref{comp:cut_integral}, we eventually get
\begin{align*}
\int_\Omega |A(U)|_1&\leq \int_{|U|_1 \geq R} |A(U)|_1 + \int_{|U|_1 < R} |A(U)|_1  \\
&\leq \frac{1}{R} \int_\Omega |A(U)|_1 |U|_1 + |\Omega| C_1(R)\\
&\leq \left(\frac{4 I \|S\|_\infty }{R}  + 1 \right)|\Omega| C_1(R),
\end{align*} 
which gives the conclusion with $C(\Omega,I\|\S\|_\infty) = 
2\, |\Omega|\, C_1(4\, I\,\|\S\|_\infty)$.
\end{proof}

\begin{cor}\emph{(\textbf{H})}\label{cor:fixed_points}
There exists a constant $C(\Omega, I\|S\|_\infty)$ depending only on $\Omega$ and $I\|S\|_\infty$ such that for any $\tau>0$ with $\rho \tau <1/2$, for all $\sigma\in [0,1]$ and for all $U\in \L^\infty(\Omega)$,
\begin{align*}
\Lambda(\sigma,u)=u \Longrightarrow \|U\|_\infty \leq \frac{1}{\alpha \tau} C(\Omega, I\|S\|_\infty). 
\end{align*}
\end{cor}
\begin{proof}
Thanks to the nonnegativeness of $U$ and the condition $\tau \rho <1/2$, inequality \eqref{ineq:fixpoint} implies 
\begin{align*}
- \Delta |A(U)|_1 \leq \frac{\sigma}{\tau} |S|_1 \leq \frac{1}{\tau} |S|_1, \quad \text{ a.e. on }\Omega .
\end{align*}
Applying Lemma \ref{Linfinityestimate} to $w= |A(U)|_1 = \displaystyle \sum_{i=1}^I a_i(U)u_i\ge0$, we get
\begin{align*}
0 \leq \sum_{i=1}^ I a_i(U) u_i \leq C(\Omega) \left(\frac{1}{\tau}I\|S\|_\infty + \int_\Omega |A(U)|_1 \right), \quad \text{ a.e. on }\Omega,
\end{align*}
and the conclusion follows thanks to Lemma \ref{lem:estil1}, the nonnegativeness of $U$, and Assumption \textbf{H2} : $a_i$ is lower bounded by $\alpha>0$.
\end{proof}

\subsubsection{End of the proof of Theorem \ref{th:existence_scheme}.}
\begin{proof}[End of the proof of Theorem \ref{th:existence_scheme}]
We just invoke Theorem \ref{th:Leray-Schauder}, and we use Lemma \ref{le:S_compact} and Corollary \ref{cor:fixed_points} to check the assumptions on $\Lambda$. 
\end{proof}

\subsection{Estimates for the scheme}\label{sec:apriori}

Applying iteratively Theorem \ref{th:existence_scheme} we get the existence of the family $(U^k)_k$ in Theorem~\ref{th:existence_scheme_it}. In particular, we already know that it satisfies for all $k\ge 1$,
\begin{align}
\label{uk1} U^0 &\ge \eta >0,\\
\label{uk2}U^k &\geq 0,\text{ and }U^k\in\L^\infty(\Omega)^I,\\
\label{uk3}A(U^k) &\in\H^2_{\nu}(\Omega)^I.
\end{align}
In this subsection, we prove that the family $(U^k)_k$ satisfies estimates \eqref{es:depend_on_tau1}--\eqref{for38}.
\subsubsection{Non uniform estimates (\emph{i.e.} depending on $\tau$)}

Let us first give a few properties of regularity for the family $(U^k)_{1\leq k\leq N-1}$.
\begin{Propo}[\textbf{H}]\label{prop:estireg}
For all $k\geq 1$ and  $p\in[1,\infty[$,
\begin{align*}
U^k &\in\mathscr{C}^0(\overline{\Omega})^I, \\
U^k &\geq \eta_{A,R,\tau}  \text{ on }\overline{\Omega},\\
A(U^k) &\in \W^{2,p}_\nu(\Omega)^I,
\end{align*}
where $\eta_{A,R,\tau}>0$ is a strictly positive constant depending on the maps $A$ and $R$ and $\tau$.
\end{Propo}

\begin{proof}
Since $R$ is continuous, for each $k\geq 0$, $R(U^k)\in\L^\infty(\Omega)^I$. We hence can write, for any positive constant $M>0$ and any $k\geq 1$
\begin{align*}
M A(U^k) - \Delta [A(U^k)] = \frac{U^{k-1}-U^k}{\tau} + R(U^k) + M A(U^k) \in \L^\infty(\Omega)^I,
\end{align*}
so that we can directly  apply Lemma \ref{lem:ell} of the Appendix 
 to get $A(U^k) \in \W^{2,p}(\Omega)^I$ for all finite values of $p$, whence, by Sobolev embedding, $A(U^k)\in\mathscr{C}^0(\overline{\Omega})^I$, and $U^k \in\mathscr{C}^0(\overline{\Omega})^I$ thanks to Assumption \textbf{H3}.
 For the lower bound on $U^k$, we will proceed by induction and prove that if $U^{k-1} \geq \ep$ for some constant $\ep>0$, then $U^k\geq \ep'$ for another constant $\ep'>0$. The result will follow by taking the minimum of the constructed finite family. Assuming hence $U^{k-1} \geq \ep >0$ (which is true by assumption for $k=1$), because of Assumption \textbf{H2}, if $M$ is large enough,  we get $M A(U^k) -\Delta [A(U^k)] \geq U^{k-1}/\tau \geq U^{k-1} \geq \ep$. We infer hence, by the maximum principle, that $A(U^k)\geq \ep/M$, whence $U^k \geq \ep/C$, where $C= M\sup_i \|a_i(U^k)\|_\infty$ (not vanishing because of Assumption  \textbf{H2}).
\end{proof}

\subsubsection{ (Uniform) $\L^1$ estimate}

We write down  the standard $\L^1$ estimate obtained by a direct integration of the equations:
\begin{Propo}[\textbf{H}]\label{prop:l1}
Assuming that $\rho\tau <1/2$, the family $(U^k)_k$ satisfies
\begin{align}
\label{ineq:l1re} \max_{0\leq k\leq N-1}\int_\Omega U^k  &\leq  2^{2\rho \tau N} \int_\Omega U^0, \\
\label{ineq:l1bisre} \sum_{k=1}^{N-1} \tau \int_\Omega \left(\rho U^k - R(U^k)\right) &\leq 2^{2\rho \tau N} \int_\Omega U^0.
\end{align}
\end{Propo}
\begin{proof}
For \eqref{ineq:l1re}, we prove in fact the more precise estimate for $k\geq 1$:
\begin{align}
\label{ineq:l1strong}\int_\Omega U^k \leq (1-\rho\tau)^{-k} \int_\Omega U^0.
\end{align}
Indeed, thanks to Assumption \textbf{H2}, \eqref{eq:scheme} implies (almost everywhere on $\overline{\Omega}$)
\begin{align}
\label{ineq:schemineq}  \frac{U^{k}-U^{k-1}}{\tau}- \Delta A(U^k) \leq \rho U^k.
\end{align}
Integrating then \eqref{ineq:schemineq} on $\Omega$, we get
 \begin{align*}
(1-\rho \tau) \int_\Omega U^k \leq \int_\Omega U^{k-1},
\end{align*}
so that \eqref{ineq:l1strong} follows by a straightforward induction. Since $\tau N=T$ and $\rho \tau <1/2$, we have $(1-\rho \tau)^{-N}\le 2^{2 \rho \tau N}$, whence \eqref{ineq:l1re}. Integrating the first equation in \eqref{eq:scheme} on $\Omega$, and summing for $1\leq k\leq N-1$, we get
 \begin{align*}
- \sum_{k=1}^{N-1} \tau \int_\Omega R(U^k) \leq \int_\Omega U^0,
\end{align*}
so that using \eqref{ineq:l1strong},
\begin{align*}
\sum_{k=1}^{N-1} \tau \int_\Omega \left(\rho U^k - R(U^k)\right) & \leq \tau\rho \sum_{k=1}^{N-1} (1-\rho\tau)^{-k} \int_\Omega U^0 +\int_\Omega U^0 \\
&= \left[\frac{\tau\rho}{1-\rho\tau} \frac{(1-\rho\tau)^{-(N-1)}-1}{(1-\rho\tau)^{-1}-1}+1\right]\int_\Omega U^0\\
&= (1-\rho\tau)^{-(N-1)} \int_\Omega U^0 \le (1-\rho\tau)^{-N} \int_\Omega U^0
\end{align*}
and \eqref{ineq:l1bisre} follows.
\end{proof} 

\subsubsection{(Uniform) Duality estimate}

We now focus on another uniform (in $\tau$) estimate for this scheme, which is reminiscent of a  duality lemma first introduced in \cite{MaPi}, which writes:
\begin{lem}
Let $\rho>0$. Let $\mu:[0,T]\times\Omega\rightarrow\R_+$ be a continuous function lower bounded by a positive constant. Smooth nonnegative solutions of the differential inequality 
\begin{align*}
\partial_t u-\Delta (\mu u)\leq \rho u \text{ on }[0,T]\times\Omega,\\
\partial_n (\mu u)=0 \text{ on }[0,T]\times\partial\Omega,
\end{align*}
 satisfy the bound
\begin{align*}
\int_{Q_T} \mu u^2\leq\exp(2 \rho T)\times \left( \Cc_\Omega^2\, \|u^0\|_{\H^{-1}_m(\Omega)}^2+\langle u^0\rangle^2\int_{Q_T} \mu\right),
\end{align*}
where $u^0:=x\mapsto u(0,x)$, $\langle u^0 \rangle$ denotes its mean value on $\Omega$ and $\Cc_\Omega$ is the Poincar\'e-Wirtinger constant.
\end{lem}
\medskip

\medskip

The proof of the previous lemma can be easily adapted from the proof of lemma A.1 in Appendix A of \cite{Lepoutre_JMPA}. We are here concerned with the following discretized version:
\begin{lem}\label{lem:duadis}
Fix $\rho>0$ and $\tau>0$ such that $\rho \tau < 1$. Let $(\mu^k)_{0\leq k\leq N-1}$ be a family of nonnegative functions 
which are integrable on $\Omega$. Consider a family $(u^k)_{0\leq k \leq  N-1}$ of nonnegative bounded functions, such that for all $k$, $\mu^k u^k \in\H^2(\Omega)$,  and which satisfies in the strong sense
\begin{align*}
\frac{u^{k}-u^{k-1}}{\tau} - \Delta(\mu^{k}u^{k}) &\leq \rho u^{k}, \text{ on }\Omega,\\
\partial_n (\mu^k u^{k}) &= 0,\text{ on }\partial\Omega.
\end{align*}
Then, this family satisfies the bound
\begin{align*}
\sum_{k=1}^{N-1}\tau\int_\Omega \mu^{k}|u^{k}|^2 \leq (1-\rho\tau)^{-2N}\left( \Cc_\Omega^2\, \|u^0\|_{\H_m^{-1}}^2+\langle u^0\rangle^2\sum_{k=1}^{N-1}\tau\int_\Omega \mu^{k}\right).
\end{align*}
\end{lem}

\begin{proof}
Since $\rho\tau<1$,  the sequence of nonnegative functions $v^k:=(1-\rho\tau)^k u^k$ satisfies
\begin{align*}
v^{k}-\frac{\tau}{1-\rho \tau}\Delta \mu^{k}v^{k}&\leq v^{k-1} ,\text{ a.e. in }\Omega,\\
\partial_n v^{k}&=0,\text{ a.e. on }\partial\Omega.
\end{align*}
Summing from $k=1$ to $k=n$ we get for all $1\leq n \leq N-1$,
\begin{align*}
v^{n}-\frac{1}{1-\rho \tau}\Delta S^{n} &\leq v^0 \text{ a.e. in }\Omega,
\end{align*}
where 
\begin{align*}
S^{n} := \tau \sum_{k=1}^n \mu^{k}v^{k}\,\in\,\H^2(\Omega).
\end{align*}
We multiply this last inequality by $\tau \mu^{n}v^{n}=S^{n}-S^{n-1}$ (with the convention $S^0=0$) and integrate over $\Omega$. We get 
$$
\tau \int_\Omega \mu^{n}|v^{n}|^2 +\frac{1}{1-\rho \tau}\int_\Omega (\nabla S^{n}-\nabla S^{n-1})\cdot \nabla S^{n}\leq \int_\Omega v^0 (S^{n}-S^{n-1}),
$$
and therefore,
$$
\tau \int_\Omega \mu^{n}|v^{n}|^2 +\frac{1}{2(1-\rho \tau)}\int_\Omega |\nabla S^{n}|^2-|\nabla S^{n-1}|^2\leq \int_\Omega v^0 (S^{n}-S^{n-1}).
$$
We sum up over $n$ again to obtain 
$$
\sum_{n=1}^{N-1}\tau\int_\Omega \mu^{n}|v^{n}|^2 +\frac{1}{2(1-\rho \tau)}\int_\Omega |\nabla  S^{N-1}|^2\leq \int_\Omega v^0  S^{N-1}.
$$
Using Poincar\'e-Wirtinger's and Young's inequalities, the right-hand side can be dominated by
\begin{align*}
\int_\Omega v^0  S^{N-1} &= \int_\Omega v^0 (S^{N-1}-\langle S^{N-1}\rangle) + \langle v^0\rangle \int_\Omega S^{N-1}\\
&\le \Cc_\Omega \|v^0\|_{\H^{-1}_m(\Omega)}\|\nabla S^{N-1}\|_2+\langle v^0\rangle \int_\Omega S^{N-1}\\
&\leq \frac{\Cc_\Omega^2\|v^0\|_{\H^{-1}_m}^2}{2} + \frac{\|\nabla S^{N-1}\|_2^2}{2}+  \langle v^0\rangle \int_\Omega S^{N-1},
\end{align*}
from which we get, since $\rho\tau <1$,
\begin{align*}
\sum_{n=1}^{N-1}\tau\int_\Omega \mu^{n}|v^{n}|^2 \leq \frac{\Cc_\Omega^2\|v^0\|_{\H^{-1}_m}^2}{2} + \langle v^0\rangle \int_\Omega S^{N-1}.
\end{align*}
Using Cauchy-Schwarz inequality one easily gets from the definition of $S^{N-1}$ 
\begin{align*}
\langle v^0 \rangle \int_\Omega S^{N-1} \leq \sqrt{\sum_{n=1}^{N-1}\tau\int_\Omega \mu^{n}|v^{n}|^2}\sqrt{\langle v^0\rangle^2\sum_{n=1}^{N-1}\tau\int_\Omega \mu^{n}},
\end{align*}
so that using another time Young's inequality, one obtains 
\begin{align*}
\sum_{n=1}^{N-1}\tau\int_\Omega \mu^{n}|v^{n}|^2\leq  \Cc_\Omega^2\|v^0\|_{\H^{-1}_m}^2 + \langle v^0\rangle^2\sum_{n=1}^{N-1}\tau\int_\Omega \mu^{n},
\end{align*}
from which we may conclude using $v^n = (1-\rho\tau)^n u^n$. 
\end{proof}
We eventually apply the previous lemma to get the following estimate:

\begin{cor}[\textbf{H}]\label{cor:dua}
Suppose $\rho\tau <1/2$. Then there exists a constant $C(\Omega,U^0,A, \rho, N\tau)$ depending only on $\Omega$, $A$, $\rho$, $N\tau$ and $\|U^0\|_{\L^1\cap\H^{-1}_m(\Omega)}$ such that the family $(U^k)_k$ satisfies
\begin{align} \label{for38re}
\sum_{k=0}^{N-1} \tau \int_\Omega \left(\sum_{i=1}^I u_i^k\right)\left(\sum_{i=1}^I a_i(U^k) u_i^k\right) \leq  C(\Omega,U^0,A, \rho, N\tau).
\end{align}
\end{cor}

\begin{proof}
In order to apply Lemma \ref{lem:duadis} to inequality \eqref{ineq:schemineq},  we introduce two families of real-valued functions  : 
\begin{align*}
u^k &:=  \sum_{i=1}^I u_i^k, \\
\mu^k &:= \frac{\sum_{i=1}^I a_i(U^k) u_i^k}{\sum_{i=1}^I u_i^k}.
\end{align*}
Notice that $\mu^k$ is a well-defined nonnegative function on $\Omega$ and $\mu^k u^k \in\H^2(\Omega)$ thanks to Proposition \ref{prop:estireg}.
Summing all coordinates of the vectorial inequality \eqref{ineq:schemineq}, we get (almost everywhere on $\Omega$)
\begin{align*}
\frac{u^k-u^{k-1}}{\tau}-\Delta(\mu^k u^k) \leq \rho u^k.
\end{align*}
 Lemma \ref{lem:duadis} yields then the following inequality: 
\begin{align*}
\sum_{k=0}^{N-1} \tau \int_\Omega \mu^k |u^k|^2 \leq 2^{4\rho N\tau} C(\Omega,u^0)\left(1+\sum_{k=0}^{N-1} \tau\int_\Omega \mu^k\right).
\end{align*}
From the definition of $\mu^k$ and $u^k$ and the continuity of the mapping $A$, it is clear that (for any $L>0$) $|u^k|\leq L$ implies $|\mu^k|\leq C(L)$, where $C(L)>0$ is some constant depending on $L$ and $A$, so that the integrals in the r.h.s. may be handled in the following way: 
\begin{align*}
\int_\Omega \mu^k &\leq C(L)|\Omega| + \int_{|u^k|>L} \mu^k \\
&\leq C(L)|\Omega| + \frac{1}{L^2}\int_{\Omega} \mu^k |u^k|^2.
\end{align*}
Then, if $L$ is large enough to satisfy  $2^{4\rho N\tau}C(\Omega,u^0)<L^2/2$, we get 
\begin{align*}
\sum_{k=0}^{N-1} \tau \int_\Omega \mu^k |u^k|^2 \leq 2 \times 2^{4\rho N\tau} C(\Omega,u^0) \left[1+ N\tau C(L)|\Omega|\right],
\end{align*}
from which we easily conclude.
\end{proof}

\section{The entropy estimate for two species}\label{sec:twospec}

 This section is devoted to the elucidation of the hidden entropy
 structure for equations of the form (\ref{eq:syst_multi_vec}). The results of this
 section are not really needed in the proof of Theorem \ref{th:globalex}, since the computations which are presented here as
 {\it{a priori}} estimates will be performed a second time at the
 level of the approximated system. 

Let us first recall that in \cite{deslepmou} was investigated a system that
took the following general form: 
\begin{align}\label{eq:system:total1}
 \partial_t u_1 - \Delta\big[(d_{11}(u_1)+ a_{12}(u_2))\, u_1\big]
 &= u_1\big(\rho_1-s_{11}(u_1)-s_{12}(u_2)\big),
 \\
\label{eq:system:total2} \partial_t u_2 - \Delta\big[(d_{22}(u_2)+ a_{21}(u_1))\,u_2\big]
 &= u_2\big(\rho_2-s_{22}(u_2)-s_{21}(u_1)\big),
\end{align}
under various conditions on the functions $d_{ij}$, $a_{ij}$, $s_{ij}$, for $i,j\in\{1,2\}$,  One of these conditions was that both $a_{12}$ and $a_{21}$ had to be increasing and concave. 
We intend here to relax this assumption in order to include several convex/concave cases. Due to the large number of possibilities, the treatment of this system in its full generality (that is, all possible functions $d_{ij}$, $a_{ij}$ and $s_{ij}$) leads to lengthy if not tedious statements and computations. To ease a little bit the understanding of this entropy structure, we will present first the specific case  of power rates coefficients 
(subsection \ref{subsec:specent}) and then treat a more general framework (subsection \ref{subsec:genent}).

\subsection{A simple specific example}\label{subsec:specent}

Since self-diffusion ($d_{ii}$ functions) usually tends to improve the 
estimates, we consider the case of constant self-diffusion rates. Similarly, since reaction terms have no real influence on the entropy structure,
 we will not consider them here.
 Consider hence the following system: 
\begin{align}
\label{eq:u}\partial_t u_1 -\Delta \Big[ u_1(d_1 + u_2^{\gamma_2}) \Big] &=   0, \\
\label{eq:v}\partial_t u_2 -\Delta \Big[ u_2(d_2 + u_1^{\gamma_1}) \Big] &= 0,
\end{align}
where $d_1,d_2,\gamma_1,\gamma_2>0$.  The entropy exhibited in \cite{deslepmou}, corresponds to the limitation : $\gamma_1 \in ]0,1[$, $\gamma_2 \in ]0,1[$ (we speak only about the control of the entropy here). Let us explain how to control an entropy in cases when $\gamma_2>1$ under the condition $\gamma_1<1/\gamma_2$. 
\medskip

We present  the following result:

\begin{Propo}
Consider $d_1>0$, $d_2>0$ and $\gamma_2>1$, $0<\gamma_1<1/\gamma_2$.  Let $(u_1, u_2)$ be a classical (that is, belonging to $\mathscr{C}^2([0,T] \times \ov{\Omega})$) positive (that is, both $u_1$ and $u_2$ are positive on $Q_T$) solution to \eqref{eq:u} -- \eqref{eq:v}  with homogeneous Neumann boundary conditions on $\partial\Omega$. Then the following {\emph{a priori}} estimates hold for $i=1,2$ and all $t\in[0,T]$
\begin{align}
\label{eslil1}\int_\Omega u_i(t)  =\int_\Omega u_i(0),
\end{align}
and 
\begin{align}
\label{esent} \sup_{t\in [0,T]}\int_\Omega u_2(t)^{\gamma_2} + C \int_0^T \int_{\Omega} \Big\{|\nabla u_1^{\gamma_1/2}|^2  + |\nabla u_2^{\gamma_2/2}|^2 + \left|\nabla (u_1^{\gamma_1/2} u_2^{\gamma_2/2})\right|^2 \Big\} &\leq C_{T,u_1(0),u_2(0)}.
\end{align}
\end{Propo}
\begin{proof}
Integrating on $[0,t[ \times\Omega$  eq. \eqref{eq:u} -- \eqref{eq:v}, we obtain estimate \eqref{eslil1}.

\vspace{2mm}

Define
\begin{align*}
h_1(t):= t-\frac{t^{\gamma_1}}{\gamma_1}, \quad h_2(t):= \frac{t^{\gamma_2}}{\gamma_2}-t, 
\end{align*}
so that both functions reach their minimum for $t=1$. Define the corresponding entropy
\begin{align*}
\E(u_1,u_2) := \frac{\gamma_1}{1-\gamma_1}\int_{\Omega} \left[ h_1(u_1)-h_1(1)\right]\, + \frac{\gamma_2}{\gamma_2-1}\int_{\Omega} \left[ h_2(u_2)-h_2(1) \right] \,.
\end{align*}
Differentiating this functional and performing the adequate integrations by parts, we get the following identity:
\begin{align*}
\frac{\dd}{\dd t} \E(u_1,u_2) &+ \gamma_1 d_1 \int_{\Omega} u_1^{\gamma_1-2} |\nabla u_1|^2  + \gamma_2 d_2 \int_{\Omega} u_2^{\gamma_2-2} |\nabla u_2|^2  \\
&+ \int_{\Omega} u_1^{\gamma_1} \,u_2^{\gamma_2} \left[\gamma_1 \left|\frac{\nabla u_1}{u_1}\right|^2 + \gamma_2 \left|\frac{\nabla u_2}{u_2}\right|^2 + 2 \gamma_1\gamma_2 \frac{\nabla u_1}{u_1}\cdot \frac{\nabla u_2}{u_2}\right] =0.
\end{align*}
Since $0<\gamma_1\,\gamma_2<1$, the quadratic form  $Q$ associated to the matrix $\begin{pmatrix}
\gamma_2 & \gamma_1 \gamma_2\\
\gamma_1\gamma_2 & \gamma_1
\end{pmatrix}$
is positive definite, whence the existence of a constant $C>0$ such as  \begin{align*}
Q(x,y)\geq C \left(x^2 + y^2\right).
\end{align*}
Using the previous inequality in the entropy identity, we get
\begin{align*}
\frac{\dd}{\dd t} \E(u_1(t),u_2(t)) &+ \gamma_1 d_1 \int_{\Omega} u_1^{\gamma_1-2} |\nabla u_1|^2  \\
&+ \gamma_2 d_2 \int_{\Omega} u_2^{\gamma_2-2} |\nabla u_2|^2  
+ C\int_{\Omega} u_1^{\gamma_1} \,u_2^{\gamma_2} \left[\,\left|\frac{\nabla u_1}{u_1}\right|^2 + \left|\frac{\nabla u_2}{u_2}\right|^2\,\right] \leq 0,
\end{align*}
from which we obtain, changing the constant $C$
\begin{align}
\label{ineq:ent}\frac{\dd}{\dd t} \E(u_1(t),u_2(t)) + C \int_{\Omega} |\nabla u_1^{\gamma_1/2}|^2  + C \int_{\Omega} |\nabla u_2^{\gamma_2/2}|^2  
+ C\int_{\Omega} \left|\nabla \sqrt{u_1^{\gamma_1} \,u_2^{\gamma_2}}\right|^2 \leq 0.
\end{align}
Since $h_1$ reaches its minimum at point $1$, we have $h_1(u_1)-h(1)\geq 0$, and integrating
 the previous inequality on $[0,t]$ leads to 
\begin{align*}
\frac{1}{\gamma_2-1}\int_\Omega u_2(t)^{\gamma_2} &+ C \int_0^t \int_{\Omega} |\nabla u_1^{\gamma_1/2}|^2  + C \int_0^t \int_{\Omega} |\nabla u_2^{\gamma_2/2}|^2 \\
 &+ C\int_0^t\int_{\Omega} \left|\nabla \sqrt{u_1^{\gamma_1} \,u_2^{\gamma_2}}\right|^2 \leq \frac{\gamma_2}{\gamma_2-1} \int_\Omega u_2(t) + \E(u_1,u_2)(0),
\end{align*}
and the conclusion follows using \eqref{eslil1}.
\end{proof}

\subsection{The general entropy structure}\label{subsec:genent}

Let us consider the general form \eqref{eq:syst_multi_vec} in the case of two species. As before, we assume that $R=0$ 
in order to simplify the presentation.

\vspace{2mm}

 We will now perform rather formal (unjustified) computations, in order to bring out the entropy structure. In Section \ref{sec:globweak}, all these computations will be rigorously 
justified (under the corresponding set of assumptions), at the (semi-)discrete level, in order to prove the existence of global weak solutions to eq. (\ref{eq:u:reac}), (\ref{eq:v:reac})  [Theorem \ref{th:globalex}].

\vspace{2mm} 

We consider a function $\Phi:\R^2\rightarrow\R$.
 If $U$ is solution of \eqref{eq:syst_multi_vec} with $R=0$, introducing $W:=\nabla \Phi(U)$, a straightforward computation (yet completely formal) leads to 
\begin{align*}
\frac{\dd}{\dd t}\int_\Omega \Phi(U)  - \int_\Omega \langle \nabla \Phi(U) , \Delta[A(U)]\rangle=0,
\end{align*}
where $\langle \cdot,\cdot \rangle$ is the inner-product on $\R^2$, whence after integration by parts, using the repeated index convention,
\begin{align*}
\int_\Omega \langle \nabla \Phi(U) , \Delta[A(U)]\rangle&= \int_\Omega (\partial_i \Phi)(U) \partial_{jj}[A_i(U)] \\
&=-\int_\Omega \partial_j[(\partial_i \Phi)(U)] \partial_{j}[A_i(U)]\\
&=-\int_\Omega \partial_k\partial_i \Phi(U) \partial_j u_k \partial_{\ell} A_i(U)\partial_j u_\ell .
\end{align*}
Summing first in the $i$ index, we recognize a matrix product and we can write 
\begin{align*}
\int_\Omega \langle \nabla \Phi(U) , \Delta[A(U)]\rangle=- \sum_{j=1}^2\int_\Omega \langle \partial_j U, \D^2\Phi(U)\D(A)(U)\partial_j U\rangle,
\end{align*}
where $\D(A)$ is the Jacobian matrix of $A$,
 and $\D^2(\Phi)$ is the Hessian matrix of $\Phi$. We eventually get 
\begin{align}
\label{eq:formalent}\frac{\dd}{\dd t}\int_\Omega \Phi(U)  + \sum_{j=1}^2\int_\Omega \langle \partial_j U, \D^2\Phi(U)\D(A)(U)\partial_j U\rangle=0.
\end{align}
This identity implies that $\Phi(U)$ becomes a Lyapunov functional of the system, as soon as $\D^2(\Phi)\D(A)$ is positive-semidefinite (in the sense of symmetric matrices, that is, its symmetric part is positive-semidefinite). The previous computation motivates the following definition
\begin{definition}[Entropy]\label{defi:entro}
Consider $D\subset\R^2$ an open set. A real valued function $\Phi\in\mathscr{C}^2(D)$ is called an \textsf{entropy} on $D$ for the system \eqref{eq:syst_multi_vec} if it is nonnegative, and satisfies, for all $X\in D$, $\D^2(\Phi)\D(A)(X)$ is positive-semidefinite (in the sense that its symmetric part is positive-semidefinite).
\end{definition}
\begin{remark}\label{rem:gradients}
 Because of the nonnegativeness of $\Phi$, one easily gets the estimate $\Phi(U)\in\L^\infty_t(\L^1_x)$ (depending on the initial data), but in fact \eqref{eq:formalent} often implies much more than this simple estimate for the solutions of the system. Indeed, in many situations, the second term in the l.h.s. of \eqref{eq:formalent} yields
estimates on the gradients of the unknowns (as in Subsection \ref{subsec:specent}). This will be of crucial importance in 
Section \ref{sec:globweak}, since no other estimate on the gradients is known for the system 
considered in this application.
\end{remark}
\begin{remark}
The set $D$ depends on the type of system that we consider. It is the set of expected values for the vector solution $U$. In our framework, we expect positive solutions, that is $D=\R_+^* \times\R_+^*$, but it is sometimes useful to consider bounded sets for $D$, this is for instance the case in \cite{jungbound}.
\end{remark}

\vspace{2mm}

 We end this section with a simple generic example, for two species, for which we do have an entropy on $\R_+^*\times\R_+^*$, and which will cover the specific example treated in Subsection \ref{subsec:specent}. As explained before, self-diffusion generally eases the study of the system. To simplify the presentation, we hence assume here again that the diffusion terms are purely of cross-diffusion type, that is $A(U)=(u_i a_i(u_j))_{i=1,2}$ with $j \neq i$. The idea is then simple, if $\det \D(A)$ and $\textnormal{Tr}\,\D(A)$ are both nonnegative, then $\D(A)$ has two nonnegative eigenvalues and is hence not far from being positive-semidefinite, in the sense that it would indeed be positive-semidefinite if it were symmetric 
; this last property may be satisfied after multiplication by a diagonal matrix, keeping the positiveness of the trace and the determinant: this will be done thanks to $\D^2(\Phi)$. More precisely, we present the following elementary proposition: 
\begin{Propo}\label{propo:ent2spec}
Consider $a_1,a_2:\R_+^*\rightarrow\R_+$ two $\mathscr{C}^1$ functions and for $X=(x_1,x_2)\in\R_+^*\times\R_+^*$ define 
\begin{align*}
A(X) &:= (x_i a_i(x_j))_{i} \quad(i\neq j),\\
\Phi(X)&: = \phi_1(x_1) + \phi_2(x_2),
\end{align*}
where $\phi_i$ is a nonnegative second primitive of $z\mapsto a_j'(z)/z$ ($i\neq j$). If $a_1,a_2$ are increasing and $\det \D(A)\geq 0$, then $\Phi$ is an entropy on $\R_+^*\times\R_+^*$ for the system \eqref{eq:syst_multi_vec}, in the sense of Definition \ref{defi:entro}. 
\end{Propo}
\begin{proof}
We compute
\begin{align*}
\D(A)(X) = 
\begin{pmatrix}
a_1(x_2)  &  x_1  a_1'(x_2) \\
x_2a_2'(x_1) & a_2(x_1)
\end{pmatrix},\quad
\D^2(\Phi)(X) = \begin{pmatrix}
\frac{a_2'(x_1)}{x_1} & 0 \\
 0 & \frac{a_1'(x_2)}{x_2}
\end{pmatrix},
\end{align*}
so that 
\begin{align*}
M(X):=\D^2(\Phi)\D(A)(X) =\begin{pmatrix}
\star & a_2'(x_1)a_1'(x_2) \\
 a_1'(x_2)a_2'(x_1) & \star
\end{pmatrix}
\end{align*}
is obviously symmetric. If the functions $a_i$ are increasing, all the coefficients of $M(X)$ are nonnegative, so that $\textnormal{Tr}\,M(X)\geq 0$ ; we also see that \begin{align*}
\det M(X) = \det \D^2(\Phi)(X) \det \D(A)(X)\geq 0,
\end{align*} which allows to conclude.
\end{proof}

\section{Global weak solutions for two species}\label{sec:globweak}

We plan in this section to prove Theorem \ref{th:globalex}. The proof is described in subsections
\ref{sub51} to \ref{sub53}.

\subsection{Scheme} \label{sub51}

{\sl{Proof of Theorem \ref{th:globalex}}}: 
In this Section, we will just invoke the results stated in Theorem~\ref{th:existence_scheme_it} of Section \ref{sec:exsch}. Since in the approximation scheme, the sequence should be initialized with a bounded function, we will introduce a sequence $(U_N^0)_N$ of bounded functions, satisfying $U_N^0\ge \eta_N>0$, and approximating $U^{in}$ in $\L^1\cap\H^{-1}_m(\Omega)\times\L^{\gamma_2}\cap\H^{-1}_m(\Omega)$. In order to apply Theorem \ref{th:existence_scheme},
 we have to exhibit the vectorial functions $A$ and $R$ and check the assumptions \textbf{H}. We define here
\begin{align}
\label{eq:R}R:\R_+^2 &\longrightarrow \R^2\\
X:= \begin{pmatrix}
 x_1\\
 x_2
 \end{pmatrix}
 &\longmapsto 
\begin{pmatrix}
\nonumber  x_1 \,(\rho_1-x_1^{s_{11}}-x_2^{s_{12}})\\
x_2 \, (\rho_2-x_2^{s_{22}}-x_1^{s_{21}})
\end{pmatrix};\\
\nonumber &\\
\label{eq:A}A:\R_+^2 &\longrightarrow \R^2\\
X&\longmapsto 
\begin{pmatrix}
\nonumber x_1 \, (d_1+x_2^{\gamma_2})\\
  x_2 \, (d_2+x_1^{\gamma_1})
\end{pmatrix}.
\end{align}
We easily check that \textbf{H1} is satisfied, and since for $X\geq 0$, $R(X)\leq \max(\rho_1,\rho_2) X$ and $A(X) \geq \min(d_1,d_2) X$, so is $\textbf{H2}$. As for \textbf{H3}, this falls within the scope of the particular case treated in 
 the Appendix (see Proposition \ref{prop:2spec}). Indeed,  
\begin{align*}
\D(A)(X) = \begin{pmatrix}
d_1+x_2^{\gamma_2} & \gamma_2 x_1 x_2^{\gamma_2-1} \\
\gamma_1 x_2 x_1^{\gamma_1-1} & d_2+x_1^{\gamma_1}\\
\end{pmatrix},
\end{align*}
and this matrix has a positive determinant for $x_1,x_2>0$, since $\gamma_1\gamma_2<1$. We therefore apply Theorem \ref{th:existence_scheme_it} (for $\tau =T/N$ small enough) in order to get the existence of a sequence $(U^k_N)_{0\leq k\leq N-1}$, which is (for $k\geq 1$) a strong solution (in the sense of Definition \eqref{def:strongsol}) of
\begin{align}
\label{eq:scheme:ex}
\frac{U^k_N-U^{k-1}_N}{\tau} -\Delta [A(U^k_N)] &= R(U^k_N)\text{ on }\Omega,\\
\pa_n A(U^k_N) &=0 \text{ on } \pa\Omega.
\end{align}

\subsection{Uniform estimates}

We aim at passing  to the limit $\tau \rightarrow 0$ in identity (\ref{eq:scheme:ex}).
 In order to do so, we need uniform (w.r.t. $\tau,N$) estimates. We recall here that thanks to Theorem \ref{th:existence_scheme_it}, we know that for all $p\in[1,\infty[$,
\begin{align}
\label{recall1}U^k&\in\mathscr{C}^0(\overline{\Omega})^2, \\
\label{recall2}U^k &\geq \eta_{A,R,N} >0,\\
\label{recall3}A(U^k)&\in\W^{2,p}_\nu(\Omega)^2,
\end{align}
and in fact, using Proposition \ref{prop:2spec}, we see that $A$ is a $\mathscr{C}^1$-diffeomorphism from $\R_+^*\times\R_+^*$ to itself whence, using \eqref{recall2} -- \eqref{recall3}, the following regularity estimate, for all $p\in[1,\infty[$,
\begin{align}
\label{recall4}U^k\in\W^{1,p}(\Omega)^2.
\end{align}
But as noticed in Remark \ref{rem:deptau},  estimates \eqref{recall1} -- \eqref{recall4} all depend on $\tau$, and we cannot use them in the passage to the limit. They will however be of great help to justify several computations on the approximated system. For instance, \eqref{recall3} allows to see that the equation defining the scheme is meaningful 
(that is, all terms are defined almost everywhere).
 
\subsubsection{Dual and $\L^1$ estimates}

Thanks to Section \ref{sec:apriori}, we already have proven three (uniform w.r.t. $\tau$) estimates : the dual estimate (\ref{for38})
 and the $\L^1$ estimates (\ref{ineq:l1}) and (\ref{ineq:l1bis}) given by Theorem \ref{th:existence_scheme_it}.

\subsubsection{Entropy estimate}

The following estimate on the sequence $(U^k_N)_{0\leq k\leq N-1}$ holds
(in this paragraph, we drop the index $N$ to ease the presentation):

\begin{Propo}[\textbf{H}]\label{propo:discrent}
There exists a constant $K_T>0$ depending only on $T$, $\Omega$ and the data of the equations ($r_i$, $d_i$, $\gamma_i$, $s_{ij}$) such that (for $\tau$ small enough), for all $N\in\N$, the corresponding sequence $(U^k)_{0\leq k\leq N-1}$ satisfying 
 \eqref{eq:scheme:ex}, also satisfies the following bound: 
\begin{align} \label{new53}
\max_{0\le \ell \le N-1}\int_\Omega (u_2^\ell)^{\gamma_2} + \sum_{k=0}^{N-1} \tau \int_{\Omega} (u_2^k)^{\gamma_2} \left\{ (u_1^k)^{s_{21}}+(u_2^k)^{s_{22}} \right\}\\
+ \sum_{k=0}^{N-1} \tau \int_{\Omega} \Big\{|\nabla (u_1^k)^{\gamma_1/2}|^2  + |\nabla (u_2^k)^{\gamma_2/2}|^2 + \left|\nabla \sqrt{(u_1^k)^{\gamma_1} \,(u_2^k)^{\gamma_2}}\right|^2 \Big\} \\
\leq K_T \,( 1+ \|u_1^{in}\|_{1}+ \|u_2^{in}\|_{\gamma_2}^{\gamma_2}).
\end{align}
\end{Propo}
\begin{proof}
We introduce as in Subsection \ref{subsec:specent} the functions
\begin{align*}
\phi_i(z) := \frac{\gamma_i}{\gamma_i-1}\left(\frac{t^{\gamma_i}}{\gamma_i}-t +1- \frac{1}{\gamma_i}\right).
\end{align*}
Since $0<\gamma_i \neq 1$, one easily checks that $\phi_i$ is a nonnegative continuous convex function defined on $\R_+$, and smooth on $\R_+^*$, so that for all $z,y>0$,
\begin{align}
\label{ineq:conv}\phi_i'(z)(z-y) \geq \phi_i(z)-\phi_i(y).
\end{align}
We define $\Phi:\R^2\rightarrow\R$ by the formula
\begin{align*}
\Phi(x_1,x_2) := \phi_1(x_1) + \phi_2(x_2).
\end{align*}
Using \eqref{recall1} -- \eqref{recall4}, we see that for all $k$, $\nabla\Phi(U^k)$ is well-defined and belongs to $ \in\W^{1,p}(\Omega)^2$ for all $p\in[1,\infty[$. We take the inner product of $\tau \nabla\Phi(U^k)$ with the vectorial equation \eqref{eq:scheme:ex} (which has a meaning a.e.). Using \eqref{ineq:conv}, we get 
\begin{align*}
\Phi(U^k)-\Phi(U^{k-1})-\tau \langle \nabla \Phi (U^k),\Delta [A(U^k)]\rangle \leq \tau\langle \nabla \Phi (U^k),R(U^k)\rangle.
\end{align*}
We now plan to reproduce (but this time at the rigorous level) the formal computation performed in Subsection \ref{subsec:genent}. Since each term of the previous inequality is (at least) integrable, we can integrate it on $\Omega$, and sum over $1\leq k\leq \ell$ in order to get 
\begin{align}
\label{ineq:entdis}\int_\Omega \Phi(U^\ell) - \sum_{k=1}^\ell \tau \int_{\Omega} \langle \nabla \Phi (U^k),\Delta [A(U^k)]\rangle \leq \sum_{k=1}^\ell \tau \int_\Omega \langle \nabla \Phi (U^k),R(U^k)\rangle + \int_\Omega \Phi(U^0).
\end{align}
Now since $\nabla \Phi(U^k) \in\W^{1,p}(\Omega)^2$  for all $p<\infty$ (see above), the following integration by parts rigorously holds
 (because $A(U^k)$ satisfies homogeneous Neumann boundary conditions)
\begin{align}
\nonumber- \int_{\Omega} \langle \nabla \Phi(U^k),\Delta A(U^k)\rangle  &= \sum_{j=1}^2\int_\Omega \langle \partial_j U^k, \D^2(\Phi)(U^k) \D(A)(U^k)\partial_j U^k\rangle \\
\label{eq:intparts}&= \int_\Omega \langle \nabla U^k, \D^2(\Phi)(U^k) \D(A)(U^k)\nabla U^k\rangle ,
\end{align} 
where $\D(A)$ is the Jacobian matrix of $A$, and $\D^2(\Phi)$ is the Hessian matrix of $\Phi$ (and the last line is a small abuse of notation). At this stage, we know, thanks to Proposition \ref{propo:ent2spec} and the very definition of $\Phi$, that the r.h.s. of \eqref{eq:intparts} is nonnegative. However, as noticed in Remark \ref{rem:gradients}, the mere lower bound by $0$ is most probably a bad estimate. Indeed, we will see that the r.h.s. of \eqref{eq:intparts} will allow us to obtain estimates for the gradient of $U^k$ (as it was the case in Subsection \ref{subsec:specent}). We can first compute $\D(A)$ (recall the definition of $A$ in \eqref{eq:A} ) and $\D^2(\Phi)$ : 
\begin{align*}
\D(A)(X) = \begin{pmatrix}
d_1+x_2^{\gamma_2} & \gamma_2 x_1 x_2^{\gamma_2-1} \\
\gamma_1 x_2 x_1^{\gamma_1-1} & d_2+x_1^{\gamma_1}\\
\end{pmatrix},
 \quad
\D^2(\Phi)(X) = \begin{pmatrix}
 \gamma_1 x_1^{\gamma_1-2}& 0 \\
0 &\gamma_2 x_2^{\gamma_2-2}
\end{pmatrix}.
\end{align*}
Writing $\D(A)(X) = \text{diag}(d_1,d_2) + M_A(X)$, we get 
\begin{align*}
 \int_\Omega \langle \nabla U^k, \D^2(\Phi)(U^k) \D(A)(U^k)\nabla U^k\rangle & =\int_\Omega \langle \nabla U^k, \D^2(\Phi) \,\text{diag}(d_1,d_2)(U^k) \nabla U^k\rangle \\
& + \int_\Omega \langle \nabla U^k, \D^2(\Phi)(U^k) M_A(U^k)\nabla U^k\rangle,
\end{align*}
and since $\D^2(\Phi)$ is diagonal,
\begin{align}
  \nonumber\int_\Omega \langle \nabla U^k, \D^2(\Phi)(U^k) \D(A)(U^k)\nabla U^k\rangle & =d_1 \gamma_1 \int_\Omega (u_1^k)^{\gamma_1-2}  |\nabla u_1^k|^2 + d_2\gamma_2 \int_\Omega (u_2^k)^{\gamma_2-2}  |\nabla u_2^k|^2\\
\label{eq:diago}& + \int_\Omega \langle \nabla U^k, S(U^k)\nabla U^k\rangle,
\end{align}
where $S(X) = \D^2(\Phi)(X)M_A(X)$. Let us now write 
\begin{align*}
S(X) &= \begin{pmatrix}
 \gamma_1 x_1^{\gamma_1-2}& 0 \\
0 &\gamma_2 x_2^{\gamma_2-2}
\end{pmatrix}\begin{pmatrix}
x_2^{\gamma_2} & \gamma_2 x_1 x_2^{\gamma_2-1} \\
\gamma_1 x_2 x_1^{\gamma_1-1} & x_1^{\gamma_1}\\
\end{pmatrix}\\
&=\begin{pmatrix}
\gamma_1 x_1^{\gamma_1-2} x_2^{\gamma_2} & \gamma_1\gamma_2 x_1^{\gamma_1-1} x_2^{\gamma_2-1} \\
\gamma_1\gamma_2 x_2^{\gamma_2-1} x_1^{\gamma_1-1} & \gamma_2 x_2^{\gamma_2-2} x_1^{\gamma_1}\\
\end{pmatrix}\\
&=x_1^{\gamma_1} x_2^{\gamma_2}\begin{pmatrix}
\gamma_1 x_1^{-2}  & \gamma_1\gamma_2 x_1^{-1} x_2^{-1} \\
\gamma_1\gamma_2 x_2^{-1} x_1^{-1} & \gamma_2 x_2^{-2}\\
\end{pmatrix},
\end{align*}
so that eventually 
\begin{align*}
S(X) &=
x_1^{\gamma_1} x_2^{\gamma_2} 
 \begin{pmatrix}
 x_1^{-1} & 0 \\
0 &  x_2^{-1}
\end{pmatrix}
\stackrel{:=L}{\overbrace{\begin{pmatrix}
\gamma_1 & \gamma_1\gamma_2  \\
\gamma_1\gamma_2  & \gamma_2 \\
\end{pmatrix}}}
 \begin{pmatrix}
 x_1^{-1} & 0 \\
0 &  x_2^{-1}
\end{pmatrix},
\end{align*}
which means that (the quadratic form associated to) the matrix $S(X)$ acts on $W:=(w_1,w_2)$  through
\begin{align*}
\langle W,S(X) W \rangle = x_1^{\gamma_1} x_2^{\gamma_2}  \langle Z, LZ\rangle,
\end{align*}
where $Z := (w_1/x_1,w_2/x_2)$. But since $\gamma_1\gamma_2<1$, the quadratic form associated to $L$ is positive definite, whence the existence of a constant $C: = C(\gamma_1, \gamma_2)>0$ such as 
\begin{align*}
\langle W,S(X) W \rangle \geq C \left(\frac{w_1^2}{x_1^2} + \frac{w_2^2}{x_2^2}  \right)x_1^{\gamma_1} x_2^{\gamma_2}.
\end{align*}
Going back to \eqref{eq:diago} and using the previous inequality, we get 
\begin{align*}
 \nonumber\int_\Omega \langle \nabla U^k, \D^2(\Phi)(U^k) \D(A)(U^k)\nabla U^k\rangle &\geq d_1 \gamma_1 \int_\Omega (u_1^k)^{\gamma_1-2}  |\nabla u_1^k|^2 + d_2 \gamma_2 \int_\Omega (u_2^k)^{\gamma_2-2}  |\nabla u_2^k|^2\\
& + C\int_\Omega \left(\frac{|\nabla u_1^k|^2}{(u_1^k)^2} +  \frac{|\nabla u_2^k|^2}{(u_2^k)^2} \right) (u_1^k)^{\gamma_1} (u_2^k)^{\gamma_2}\\
& \geq C \int_{\Omega} \Big\{|\nabla (u_1^k)^{\gamma_1/2}|^2  + |\nabla (u_2^k)^{\gamma_2/2}|^2   \Big\}\\
&+ C \int_\Omega \left|\nabla \sqrt{(u_1^k)^{\gamma_1} \,(u_2^k)^{\gamma_2}}\right|^2 ,
\end{align*}
where we changed the constant $C$ in the last line. If we call $\Gamma_k$ the r.h.s. of the previous equality, we sum up the previous estimates \eqref{ineq:entdis} and write
\begin{align*}
\int_\Omega \Phi(U^\ell) + \sum_{k=1}^\ell \tau \Gamma_k  \leq C \left(\sum_{k=1}^\ell \tau \int_\Omega \langle \nabla \Phi (U^k),R(U^k)\rangle + \int_\Omega \Phi(U^0)\right).
\end{align*}
Since $\U^0$ (=$U^0_N$) approaches $U^{in}$ in $\L^1(\Omega)\times\L^{\gamma_2}(\Omega)$ and $|\Phi(x_1,x_2)|\lesssim 1+|x_1| + |x_2|^{\gamma_2}$, it is easy to check that $\| \Phi(U^0_N) \|_1 \lesssim 1+\|u_1^{in}\|_1+\|u_2^{in}\|_{\gamma_2}^{\gamma_2}$ up to some irrelevant constant (independent of $N$ of course). On the other hand, $\phi_1\geq 0$ and $z^{\gamma_2} \lesssim \phi_2(z) + z $. Using the $\L^1$ estimate \eqref{ineq:l1} given by Theorem \ref{th:existence_scheme_it}, we infer eventually
\begin{align*}
\int_\Omega (u_2^\ell)^{\gamma_2} + \sum_{k=1}^\ell \tau \Gamma_k  \leq C \left( \sum_{k=1}^\ell \tau \int_\Omega \langle \nabla \Phi (U^k),R(U^k)\rangle + 1 + \|u_1^{in}\|_1 + \|u_2^{in}\|_{\gamma_2}^{\gamma_2}\right).
\end{align*}
\medskip

Obtaining the desired estimate now
reduces to handling the reaction terms. Notice that from the definition of $R$,  $R(X) = (\rho_1 x_1,\rho_2 x_2) - R^{-}(X)$,
 with $R^{-}(X)\geq 0$ given by
\begin{align*}
R^{-}(X) = 
\begin{pmatrix}
\nonumber  x_1(x_1^{s_{11}}+x_2^{s_{12}})\\
x_2(x_2^{s_{22}}+x_1^{s_{21}})
\end{pmatrix}.
\end{align*}
Since $\rho=\max(\rho_1,\rho_2)$, estimate \eqref{ineq:l1bis} of Theorem \ref{th:existence_scheme_it} implies 
\begin{align}
\label{ineq:R-}\sum_{k=1}^{N-1}\tau\int_\Omega R^-(U^k) \leq C(\|u_1^{in}\|_1 + \|u_2^{in}\|_1).
\end{align}

Using the definition of $\phi_i$, one easily checks that for all $x_i>0$, $x_i \phi'_i(x_i) \lesssim \phi_i(x_i) + 1$, up  to some irrelevant constant.
 We hence get 
\begin{align*}
\int_\Omega (u_2^\ell)^{\gamma_2} + \sum_{k=1}^\ell \tau \Gamma_k  &\leq C\bigg[\sum_{k=1}^\ell  \tau\int_\Omega \Phi(U^k) + 1 + \|u_1^{in}\|_1 + \|u_2^{in}\|_{\gamma_2}^{\gamma_2}\bigg] \\
&-C  \sum_{k=1}^\ell \tau \int_\Omega \langle \nabla \Phi (U^k),R^{-}(U^k)\rangle  ,
\end{align*}
which, using $\phi_1(z) \lesssim 1+z$ and $\phi_2(z)\lesssim 1+z^{\gamma_2}$ together with estimate \eqref{ineq:l1}, may be written
\begin{align*}
\int_\Omega (u_2^\ell)^{\gamma_2} + \sum_{k=1}^\ell \tau \Gamma_k  &\leq C\bigg[\sum_{k=1}^\ell  \tau\int_\Omega (u_2^k)^{\gamma_2} + 1 + \|u_1^{in}\|_1 + \|u_2^{in}\|_{\gamma_2}^{\gamma_2}\bigg] \\
&-C  \sum_{k=1}^\ell \tau \int_\Omega \langle \nabla \Phi (U^k),R^{-}(U^k)\rangle.
\end{align*}
Expanding $\langle \nabla\Phi(X),R^-(X)\rangle$, we get
\begin{align*}
\langle \nabla\Phi(X),R^-(X)\rangle = \phi_1'(x_1) x_1(x_1^{s_{11}}+x_2^{s_{12}}) + \phi_2'(x_2) x_2(x_2^{s_{22}}+x_1^{s_{21}}).
\end{align*}
Since $\gamma_2>1$, we can write $\phi_2'(x_2) x_2 = c_2 (x_2^{\gamma_2}-x_2)$ with $c_2>0$. Furthermore, if $x_1\geq 1$, one easily checks that $\phi_1'(x_1) \geq 0$, and on the other hand, $x_1\mapsto x_1\phi_1'(x_1)$ continuously extends to $\R_+$ and is hence lower bounded for $x_1\leq 1$ by some negative constant $m_1$. All in all, we get for $X\in\R_+^2$ the estimate
\begin{align*}
\langle \nabla\Phi(X),R^-(X)\rangle \geq m_1(x_1^{s_{11}}+x_2^{s_{12}}) +  c_2 (x_2^{\gamma_2}-x_2)(x_2^{s_{22}}+x_1^{s_{21}}),
\end{align*}
whence 
\begin{align*}
c_2 x_2^{\gamma_2}(x_2^{s_{22}}+x_1^{s_{21}})-\langle \nabla\Phi(X),R^-(X)\rangle \leq |m_1|(x_1^{s_{11}}+x_2^{s_{12}}) +  c_2 x_2(x_2^{s_{22}}+x_1^{s_{21}}).
\end{align*}
Remember now that $s_{11}<1$ and $s_{12}<\gamma_2+s_{22}/2$. The last inequality implies in particular the existence of a constant $C>0$ such as $|m_1|x_2^{s_{12}}\leq C+ c_2 x^{\gamma_2+s_{22}}/2$, so that we have using $x_1^{s_{11}}\leq 1 + x_1$
\begin{align*}
\frac{c_2}{2} x_2^{\gamma_2}(x_2^{s_{22}}+x_1^{s_{21}})-\langle \nabla\Phi(X),R^-(X)\rangle \leq |m_1|(1+x_1) +  c_2 x_2(x_2^{s_{22}}+x_1^{s_{21}}),
\end{align*}
and we may eventually write (changing the constant $C$ again)
\begin{align}
\label{ineq:togron}\int_\Omega (u_2^\ell)^{\gamma_2} + \sum_{k=1}^\ell \tau \int_\Omega (u_2^k)^{\gamma_2}\left\{ (u_1^k)^{s_{21}} + (u_2^k)^{s_{22}} \right\}  + \sum_{k=1}^\ell \tau \Gamma_k\\
\leq C \left( \sum_{k=1}^\ell  \tau\int_\Omega (u_2^k)^{\gamma_2} + 1 + \|u_1^{in}\|_1 + \|u_2^{in}\|_{\gamma_2}^{\gamma_2} \right),
\end{align}
where we used estimates \eqref{ineq:l1} and \eqref{ineq:l1bis} (with its consequence \eqref{ineq:R-}) to get the estimate 
\begin{align*}
 \sum_{k=1}^\ell  \tau\int_\Omega u_1^k + \sum_{k=1}^\ell  \tau\int_\Omega  u_2^k \{(u_2^k)^{s_{22}}+(u_1^k)^{s_{21}}\}  \leq C(\|u_1^{in}\|_1 + \|u_2^{in}\|_1).
\end{align*}
\medskip
We may now conclude using a discrete Gronwall Lemma. Indeed, if we call $w_\ell$ the first integral in the l.h.s. of \eqref{ineq:togron} and define $w_0:=C(1+\|u_1^{in}\|_1+\|u_2^{in}\|_{\gamma_2}^{\gamma_2})$, we have (since $\Gamma_k\geq 0$)
\begin{align*}
(1-C\tau)w_\ell &\leq C\tau \sum_{k=1}^{\ell-1}  w_k + w_0,
\end{align*}
whence, as soon as $C\tau <1$, 
\begin{align*}
(1-C\tau)^{\ell}\, w_\ell \leq C\tau\, \sum_{k=1}^{\ell-1} (1-C\tau)^{\ell-1}\, w_k
 +  (1-C\tau)^{\ell -1} \,w_0,
\end{align*}
\medskip
from which we get by a straightforward induction 
\begin{align*}
w_\ell \leq \frac{w_0}{(1-C\tau)^{\ell}}\, ,
\end{align*}
and the conclusion of the proof of Proposition \ref{propo:discrent} easily follows
 from this last estimate.
\end{proof}

\subsection{Passage to the limit} \label{sub53}

 We come back to the proof of Theorem \ref{th:globalex}. We introduce at this level the 

\begin{definition}\label{def:pro}
For a given family $h:=(h^k)_{0\leq k \leq N-1}$ of functions defined on $\Omega$, we denote by $\underline{h}^{N}$ the step (in time) function defined on $\R \times \Omega$ by
\begin{align*}
\underline{h}^{N}(t,x):= \sum_{k=0}^{N-1} h^k(x) \mathbf{1}_{]k\tau,(k+1)\tau]} (t).
\end{align*}
We then have by definition, for all $p,q\in[1,\infty[$,
\begin{align*}
\|\underline{h}^{N}\|_{\L^q\big([0,T];\L^p(\Omega)\big)}=\left(\sum_{k=0}^{N-1} \tau \|h^{k}\|^q_{\L^p(\Omega)}\right)^{1/q},
\end{align*}
and in particular
\begin{align*}
\|\underline{h}^{N}\|_{\L^p(Q_T)}=\left(\sum_{k=0}^{N-1} \tau \int_\Omega |h^{k}(x)|^p\dd x\right)^{1/p}. 
\end{align*}
\end{definition}
Using an analogous notation for the family of vectors $(U^k_N)_{0\le k\le N-1}$, one  easily checks
that equation \eqref{eq:scheme:ex} can be rewritten as (since the functions are extended by $0$ on $\R_-$)
\begin{align} \label{ptl}
\partial_t \underline{U}^N &= \sum_{k=1}^{N-1} (U^k_N-U^{k-1}_N) \otimes \delta_{t^{k}}  + U^0_N \otimes \delta_0  \text{ in } \mathscr{D}'(]-\infty,T[\times\Omega)^2 \\
\label{eq:scheme:dis}  &= \sum_{k=1}^{N-1}\tau  ( (\Delta[A(U^k)] + R(U^k)) \otimes \delta_{t^{k}}  + U^0_N \otimes \delta_0  \text{ in } \mathscr{D}'(]-\infty,T[\times\Omega)^2 ,
\end{align}
where $t^k =k\tau$. We intend to pass to the limit $N=1/\tau\rightarrow \infty$ 
in eq. (\ref{eq:scheme:dis}).
\medskip

\vspace{2mm}

In order to do so, let us recall the bounds (all are uniform w.r.t. $N$) obtained so far:
 \begin{align}
\label{es:l1}\underline{u_1}^N+\underline{u_2}^N &\din \L^\infty_t(\L^1_x),\\
\label{es:dua1} (\underline{u}_1^N + \underline{u}_2^N)((d_1+(\underline{u}_2^N)^{\gamma_2})\underline{u}_1^N + (d_2+(\underline{u}_1^N)^{\gamma_1})\underline{u}_2^N) &\din \L^1_{t,x}, \\
\label{es:dua2} \underline{u_1}^N+\underline{u_2}^N &\din\L^2_{t,x},\\
\label{es:ent} \underline{u}_2^N &\din\L^\infty_t(\L^{\gamma_2}_x),\\
\label{es:entreac} (\underline{u}_2^N)^{\gamma_2} \left\{ (\underline{u}_1^N)^{s_{21}} + (\underline{u}_2^N)^{s_{22}} \right\}  &\din\L^{1}_{t,x},\\
\label{es:grad1} \nabla(\underline{u}_1^N)^{\gamma_1/2} &\din\L^2_{t,x}, \\
\label{es:grad2} \nabla(\underline{u}_2^N)^{\gamma_2/2} &\din\L^2_{t,x},
\end{align} 
where \eqref{es:l1} is a consequence of estimate (\ref{ineq:l1}),
 \eqref{es:dua1} is a consequence of estimate (\ref{for38}) (both in Theorem \ref{th:existence_scheme_it}), \eqref{es:dua2} is a consequence of \eqref{es:dua1} (each term is nonnegative) and \eqref{es:ent} -- \eqref{es:grad2} are all consequences of estimate (\ref{new53}) in
 Proposition \ref{propo:discrent}. Then, thanks to \eqref{es:entreac} and \eqref{es:dua2},
 we see that
\begin{align}
\underline{u}_i^N \din \L^{\gamma_i^+}_{t,x}. 
\end{align}

\medskip

This  means that $(\underline{u}_i^N)^{\gamma_i/2}\din\L^{2^+}_{t,x}$, so that using \eqref{es:dua1} and writing for $i\neq j$
\begin{align*}
\underline{u}_j^N (\underline{u}_i^N)^{\gamma_i} = \stackrel{\din\L^2_{t,x}}{\overbrace{\underline{u}_j^N (\underline{u}_i^N)^{\gamma_i/2}}} \stackrel{\din\L^{2^+}_{t,x}}{\overbrace{(\underline{u}_i^N)^{\gamma_i/2}}},
\end{align*}
we get
\begin{align}
\label{es:lapla}A(\underline{U}^N) \din\L^{1^+}_{t,x}\times \L^{1^+}_{t,x}.
\end{align}
As for the reaction terms, the coefficients $s_{ij}$ are precisely chosen in such a way that the corresponding nonlinearities may all be handled. Indeed, $s_{11} <1$, so that $(\underline{u}_1^N)^{s_{11}+1}$ is bounded in $\L^{1^+}_{t,x}$ using \eqref{es:dua2}. Then  $(\underline{u}_2^N)^{s_{22}+1}$ is bounded in $\L^{1^+}_{t,x}$ using \eqref{es:entreac}. Also, since $s_{21}<2$,  we see that 
\begin{align*}
\underline{u}_2^N (\underline{u}_1^N)^{s_{21}} &= \stackrel{\din\L^{\gamma_2}_{t,x}}{\overbrace{\underline{u}_2^N (\underline{u}_1^N)^{s_{21}/\gamma_2}}} \stackrel{\din\L^{(\gamma_2')^+}_{t,x}}{\overbrace{(\underline{u}_1^N)^{s_{21}/\gamma_2'}}} \din\L^{1^+}_{t,x},
\end{align*}
where $1/\gamma_2+1/\gamma_2'=1$. Now if $s_{12}\le \gamma_2/2$, we know from \eqref{es:dua1} that $\underline{u}_1^N (\underline{u}_2^N)^{s_{12}} $ is bounded in $\L^{1^+}_{t,x}$. 
Otherwise, $\gamma_2/2<s_{12}<\gamma_2+s_{22}/2$, and we use \eqref{es:dua1} and \eqref{es:entreac} in order to get
 \begin{align*}
\underline{u}_1^N (\underline{u}_2^N)^{s_{12}} &= \stackrel{\din\L^{2}_{t,x}}{\overbrace{\underline{u}_1^N (\underline{u}_2^N)^{\gamma_2/2}}} \stackrel{\din\L^{2^+}_{t,x}}{\overbrace{(\underline{u}_2^N)^{s_{12}-\gamma_2/2}}} \din\L^{1^+}_{t,x}.
\end{align*}
Finally, 
\begin{align}
\label{es:reac}R(\underline{U}^N) \din\L^{1^+}_{t,x}\times\L^{1^+}_{t,x}.
\end{align}
The previous bounds allow (at least) to obtain (up to extraction of some subsequence) $\L^1_{t,x}$ weak convergence (thanks to Dunford-Pettis Theorem) for $(\underline{U}^N)_N$, $(A(\underline{U}^N))_N$ and $(R(\underline{U}^N))_N$.  The strategy is then classical: one has
 to commute the weak limits and non-linearities, possibly by proving some  strong compactness. Estimates \eqref{es:grad1} -- \eqref{es:grad2} are of course very helpful in this situation, since they show that  oscillations w.r.t. the $x$ variable cannot develop.

\vspace{2mm}

\textsf{\emph{But}} since we kept in our assumptions the possibility that $\gamma_2>2$, estimate \eqref{es:grad2} degenerates. Indeed
$\nabla f = \frac{2}{\gamma_2}  f^{\frac{2-\gamma_2}{2}} \nabla f^{\gamma_2/2}$,
and for small values of $f$, no information on $\nabla f$ can  be recovered from $\nabla ( f^{\gamma_2/2})$. This type of situation is frequent in the study of degenerate parabolic equation (such as the porous medium equation for instance) and the usual Aubin-Lions Lemma 
cannot directly be applied. Notice that for $(\underline{u}_1^N)_N$, there is no such issue : \eqref{es:grad1} automatically yields an estimate on $\nabla \underline{u}_1^N$ (this is the strategy used to recover compactness in \cite{Chen2006,deslepmou} for instance ) : 
$\nabla \underline{u}_1^N = \frac{2}{\gamma_1}  (\underline{u}_1^N)^{\frac{2-\gamma_1}{2}} \nabla (\underline{u}_1^N)^{\gamma_1/2}$
is bounded (at least) in $\L^1_{t,x}$. Indeed, $(\underline{u}_1^N)^{\frac{2-\gamma_1}{2}}$ is bounded in $\L^{4/(2-\gamma_1)}_{t,x}$ thanks to \eqref{es:dua2} and $\nabla [(\underline{u}_1^N)^{\gamma_1/2}]$ is bounded in $\L^{4/(2+\gamma_1)}_{t,x}$ since it is bounded in $\L^2_{t,x}$ (because of \eqref{es:grad1}). Furthermore, let us write $S_\tau(\underline{U}^N):(t,x)\mapsto \underline{U}^N(t-\tau,x)$. Using \eqref{es:lapla} and \eqref{es:reac} to get
\begin{equation*}
\frac{\underline{U}^N-S_\tau \underline{U}^N}{\tau} =\Delta [A(\underline{U}^N)] + R(\underline{U}^N) \din \L^1_t(\W^{-2,1}_x)\times\L^1_t(\W^{-2,1}_x),
\end{equation*}
one can then apply a discrete version of Aubin-Lions lemma to $(\underline{u}_1^N)_N$ (see for instance \cite{drejun}) to recover strong compactness for $(\underline{u}_1^N)_N$ in $\L^1_{t,x}$.

\vspace{2mm}

To prove that $(\underline{u}_2^N)_N\ddin\L^2_{t,x}$, we first evaluate \eqref{eq:scheme:dis}  on some test function $\Psi\in\mathscr{D}(]-\infty,T[\times\Omega)^2$, to get 
\begin{align}
\nonumber\langle \partial_t \underline{U}^N , \Psi \rangle_{\mathscr{D}',\mathscr{D}}  &= \sum_{k=1}^{N-1}\tau \int_\Omega \langle \Delta[A(U_N^k)]+R(U_N^k),\Psi(t^{k})\rangle + \int_\Omega \langle U^0_N, \Psi(0)\rangle\\
\label{eq:lim}&= \sum_{k=1}^{N-1} \tau \int_\Omega \large\langle A(U_N^k),\Delta\Psi(t^{k})\large\rangle +  \tau \int_\Omega \large\langle R(U_N^k),\Psi(t^{k})\large\rangle + \int_\Omega \langle U^0_N, \Psi(0)\rangle.
\end{align}
Using estimates \eqref{es:lapla}, \eqref{es:reac} and \eqref{es:l1}, we hence have (using $N\tau = T$)
\begin{align*}
 \left|\langle \partial_t \underline{U}^N , \Psi \rangle_{\mathscr{D}',\mathscr{D}} \right|
\leq C_T \|\Psi\|_{\L^\infty_t(\H^{L}_x)},
\end{align*}
where $L$ is a sufficiently large integer. Using this estimate together with \eqref{es:grad2},
 we may apply Lemma 4.1 of \cite{mou} to get $(u_2^N)_N \ddin\L^2_{t,x}$. Then, up to the extraction of a subsequence,
 \begin{equation}\label{cv:strong}
 \underline{U}^N \operatorname*{\longrightarrow}_{N\rightarrow\infty} U \qquad \text{ in } \L^1([0,T]\times\Omega)\times \L^2([0,T]\times\Omega).
 \end{equation}

\medskip

 This strong convergence ensures that the weak limits of $(A(\underline{U}^N))_N$ and $(R(\underline{U}^N))_N$ are respectively $A(U)$ and $R(U)$.

\vspace{2mm}

We now can go back to \eqref{eq:lim}, and write it as
\begin{align*}
-\int_0^T \int_\Omega \langle \underline{U}^N , \partial_t \Psi \rangle  &= \sum_{k=1}^{N-1}\tau \int_\Omega \langle \Delta[A(U_N^k)]+R(U_N^k),\Psi(t^{k})\rangle + \int_\Omega \langle U^0_N, \Psi(0)\rangle,
\end{align*}
so that a straightforward density argument allows to replace $\Psi$ by some test function $\Psi\in\mathscr{C}^1_c([0,T[;\mathscr{C}^2_{\nu}(\overline{\Omega}))^2$. 
\medskip

We get
\begin{align*}
-\int_0^T \int_\Omega \langle \underline{U}^N , \partial_t \Psi \rangle  &= \int_\tau^T \int_\Omega \large\langle A(\underline{U}^N),\Delta\widetilde{\Psi^N}\large\rangle +  \int_\tau^T\int_\Omega \large\langle R(\underline{U}^N),\widetilde{\Psi^N}\large\rangle + \int_\Omega \langle U^0_N, \Psi(0)\rangle,
\end{align*}
where 
\begin{align*}
\widetilde{\Psi}^N(t,x) := \sum_{k=1}^{N-1} \Psi(t^{k},x) \mathbf{1}_{]t^{k},t^{k+1}]}(t) \operatorname*{\longrightarrow}_{N\rightarrow\infty}^{\L^\infty([0,T]\times\Omega)} \Psi, 
\end{align*}
and we have the same convergence for $\Delta \widetilde{\Psi}^N$ towards $\Delta \Psi$.  We now  know that the three sequences $(\underline{U}^N)_N$, $(A(\underline{U}^N))_N$ and $(R(\underline{U}^N))_N$ converge weakly (up to a subsequence) in $\L^1_{t,x}$,  so that the three first integrals of the equality will converge to the expected quantities, that is,
\begin{align}
\label{cv:weak_form1}-\int_0^T \int_\Omega \langle \underline{U}^N , \partial_t \Psi \rangle &\operatorname*{\longrightarrow}_{N\rightarrow\infty} -\int_0^T \int_\Omega \langle U , \partial_t \Psi \rangle, \\
 \int_\tau^T \int_\Omega \large\langle A(\underline{U}^N),\Delta\widetilde{\Psi^N}\large\rangle &\operatorname*{\longrightarrow}_{N\rightarrow\infty} \int_0^T \int_\Omega \large\langle A(U),\Delta\Psi\large\rangle, \\
\label{cv:weak_form3} \int_\tau^T\int_\Omega \large\langle R(\underline{U}^N),\widetilde{\Psi^N}\large\rangle &\operatorname*{\longrightarrow}_{N\rightarrow\infty}  \int_0^T\int_\Omega \large\langle R(U),\Psi\large\rangle,
\end{align}
whereas for the initial datum,
\begin{align}\label{cv:weak_form4}
\int_\Omega \langle U^0_N, \Psi(0)\rangle \operatorname*{\longrightarrow}_{N\rightarrow\infty} \int_\Omega \langle U^{in},\Psi(0)\rangle,
\end{align}
thanks to 
the fact that $(U^0_N)_N$ approaches $U^{in}$ in $\L^1(\Omega)\times\L^{\gamma_2}(\Omega)$. We have proved that $U$ is a nonnegative \emph{local} (in time) (very) weak solution to \eqref{eq:u:reac}--\eqref{eq:uv_neum} on $[0,T]\times \Omega$ (that is, the weak formulation \eqref{def:weak_formul1}--\eqref{def:weak_formul2} is satisfied for all $\psi_1,\,\psi_2$ in $\mathscr{C}^1_c([0,T),\mathscr{C}^2_\nu(\overline{\Omega}))$).

\medskip
Let us now show that we can extend $U$ on $\R_+\times\Omega$ so that it gives a global (in time) solution. To do this, we make appear explicitly the dependency in $\tau$ (and then indirectly in $T=\tau N$) of our semi-discrete approximation : we write $\underline{h}^N_\tau$ the function $\underline{h}^N$ defined in Definition \ref{def:pro}. Notice that it is then clear that given an infinite sequence $(h^k)_{k\in \N}$ of functions defined on $\Omega$, for all $m\in \N-\{0\}$ the function $\underline{h}_{\tau}^{mN}$ is well defined on $\R\times\Omega$ and it coincides with $\underline{h}_{\tau}$ on $[0,mT]\times\Omega$, where $\underline{h}_{\tau}(t,\cdot):=\sum_{k=0}^{\infty} h^k \mathbf{1}_{]k\tau,(k+1)\tau]} (t)$. Applying iteratively Theorem \ref{th:existence_scheme}, we get the existence of an infinite sequence $(U^k)_{k\in \N}$ solving \eqref{eq:scheme} and satisfying \eqref{es:depend_on_tau1}--\eqref{for38} with $N$ replaced by any $N'\ge N$. Then $\underline{U}_{\tau}^{mN}$ is defined for all $m\in\N-\{0\}$ and it furthermore coincides with $\underline{U}_\tau$ on $[0,mT]\times\Omega$. Extracting subsequences, we can perform the proof of convergence on $[0, 2T ]$, $[0, 3T ]$, ..., so that by Cantor’s diagonal argument, we get that convergence \eqref{cv:strong}  (together with the existence of the limit $U$) and convergences \eqref{cv:weak_form1}--\eqref{cv:weak_form4} hold true with $T$ replaced by $m T$ and $\underline{U}^N$ replaced by $\underline{U}_\tau^{mN}$ (or equivalently by $\underline{U}_\tau$), for any $m\in \N-\{0\}$ and for $\Psi$ any test function in $\mathscr{C}^1_c([0,mT[;\mathscr{C}^2_{\nu}(\overline{\Omega}))^2$. At the end of the day, $U$ is defined in $\L^1_{\textnormal{loc}}(\R_+,\L^1(\Omega))\times \L^2_{\textnormal{loc}}(\R_+,\L^2(\Omega))$ and is a global (in time) (very) weak solution to \eqref{eq:u:reac}--\eqref{eq:uv_neum}.

\medskip
To conclude the proof, it suffices to show estimates \eqref{eq:estimationL2}, \eqref{24n}, \eqref{nne} for any $s>0$. This is done by passing to the limit $N\rightarrow\infty$ in estimates \eqref{ineq:l1}, \eqref{for38} and \eqref{new53}, with $T$ and $N$ replaced by some $mT>s$ and $mN$. We use the strong convergence of $\underline{U}_\tau^N$ in $\L^1([0,mT]\times\Omega)$ and Fatou's lemma to compute the limits in \eqref{for38} and the two first integrals in \eqref{new53}. To compute the limits of the remaining terms in \eqref{new53}, we notice that $(\underline{u}_{\tau,1}^N)^{\gamma_1/2},\,(\underline{u}_{\tau,2}^N)^{\gamma_2/2},\, (\underline{u}_{\tau,1}^N)^{\gamma_1/2} (\underline{u}_{\tau,2}^N)^{\gamma_2/2}\din \L^{2+}([0,mT]\times\Omega)$ (thanks to estimates \eqref{es:dua1} and \eqref{es:entreac}), hence the weak convergence of these sequences in $\L^2([0,mT]\times\Omega)$, and use the weak lower semi-continuity of the norm in $\L^2([0,mT]\times\Omega)$ on the sequences of the gradients. To get \eqref{24n}, we first use again the strong convergence of $\underline{U}_\tau^N$ in $\L^1([0,mT]\times\Omega)$ and Fatou's lemma to compute the limit in \eqref{ineq:l1}, which gives that $U$ is in $\L^\infty_{\textnormal{loc}}(\L^1(\Omega))$. It does not give directly the very estimate \eqref{24n}, but it is sufficient to compute rigorously for $i\neq j$ and for almost every $s\in\R_+$ (by taking in identities \eqref{def:weak_formul1}--\eqref{def:weak_formul2} a sequence of functions $\psi_i$ which are uniform in space, $\mathscr{C}^1_c$ in time, uniformly bounded in $\L^\infty(\R_+)$ and approximate the function $\mathbf{1}_{t\in[0,s]}$ in $\textnormal{BV}(\R_+)$, that is the sequence of $\psi_i$ approximates $\mathbf{1}_{t\in[0,s]}$ in $\L^1_{\text{loc}}(\R_+)$ and the sequence of the derivatives $\pa_t \psi_i$ approximate $\pa_t\mathbf{1}_{t\in[0,s]}=\delta_s-\delta_0$ weakly in the sense of Radon measures on $\R_+$)
\begin{equation*}\begin{split}
 \int_{\Omega} u_i(s,x) \, \dd x\
= \int_0^{s}\int_{\Omega} u_i(t,x)\,\big(\rho_i-u_i(t,x)^{s_{ii}}-u_j(t,x)^{s_{ij}}\big) \, \dd x\,\dd t \le \rho_i \int_0^{s}\int_{\Omega} u_i(t,x)\, \dd x\,\dd t,
\end{split}\end{equation*}
and we conclude using a Gronwall's lemma.

\section{Appendix}\label{sec:app}

\subsection{Examples of systems satisfying \textbf{H3}}\label{subsec:satisfH3}

In this section, we provide  sufficient conditions on the functions $a_i:\R_+^I\rightarrow\R_+$, allowing to prove that 
\begin{align*}
A: \R_+^I &\longrightarrow \R_+^I\\
X:=\begin{pmatrix}
x_1\\
\vdots\\
x_I
\end{pmatrix}&\longmapsto 
\begin{pmatrix}
a_1(X)\,x_1\\
\vdots\\
a_I(X)\,x_I
\end{pmatrix}
\end{align*}
is a homeomorphism from $\R_+^I$ to itself. More precisely, in our framework assumptions \textbf{H1} and \textbf{H2} are satisfied for the functions $a_i$ (that is, continuity and positive lower bound) and we assume the existence of a convex entropy. This last property implies in particular that $A$ is non-singular with $\det \D (A)>0$. We investigate two cases where these assumptions allow to prove that $A$ is a homeomorphism, that is, \textbf{H3} is satisfied.

\subsubsection{Two species with increasing diffusions}\label{subsubsec:2spec}

We start here with the case when $I=2$ (two species).

\begin{Propo}\label{prop:2spec}
Assume that $a_1,a_2:\R_+^2\rightarrow\R_+$ are continuous functions, lower bounded by $\alpha>0$. Assume that on $\R_+\times\R_+$, $A$ is strictly increasing (that is, each component is strictly increasing w.r.t. each of its variables) and that on $\R_+^*\times\R_+^*$, $A$ is $\mathscr{C}^1$ and $\det\D(A)$ remains strictly positive. Then $A$ is a homeomorphism from $\R_+^2$ to itself and a $\mathscr{C}^1$-diffeomorphism from $(\R_+^*)^2$ to itself.
\end{Propo}

\begin{proof}
It suffices to prove that $A$ is a bijection from $\R_+^2$ to itself. Then, the inverse function theorem ensures that $A$ is a diffeomorphism on $(\R_+^\ast)^2$. Thanks to the positive lower bound for $a_1$ and $a_2$ and the continuity of $A$ on $(\R_+)^2$, it is easy to check that $A^{-1}$ is continuous on $(\R_+)^2$.

Let us fix $(f,g) \in\R_+^2$ and find $(u,v)\in\R_+^2$ such that $A(u,v) = (f,g)$, that is 
\begin{align*}
\begin{pmatrix}
A_1(u,v)\\
A_2(u,v)
\end{pmatrix} =
\begin{pmatrix}
a_1(u,v) \,u\\
a_2(u,v) \,v
\end{pmatrix} = \begin{pmatrix}
f\\g
\end{pmatrix}.
\end{align*}
We first solve the first equation, considering $u$ as the unknown: for any $v\ge0$, the function $A_1(\cdot,v): u\in\R_+\mapsto a_1(u,v)\, u\in\R_+$ is strictly increasing (by assumption) and onto (due to the continuity and positivity of $a_1(\cdot,v)$). Therefore it is a bijection: we write $u=u_f(v)$ the only solution in $\R_+$ of $a_1(u,v)\,u=f$.

The monotonicity of $A_1$ (in $u$ and $v$) implies that $u_f$ is strictly decreasing on $\R_+$. This together with the continuity of $A_1$ on $\R_+^2$ implies that $u_f$ belongs to the class $\mathscr{C}^0(\R_+)$: indeed, for any $v\ge0$, if $v_n$ is a decreasing (resp. increasing) sequence converging to $v$, then $u_f(v_n)$ is increasing (resp. decreasing) and upper (resp. lower) bounded by $u_f(v)$, therefore it converges to some limit $l$ satisfying $A(l,v)=\lim A(u_f(v_n),v_n)=f$, that is, $l=u_f(v)$. Furthermore, we have on $(\R_+^\ast)^2$, $\det\D(A)=[\pa_u A_1][\pa_v A_2]-[\pa_1 A_2][\pa_2 A_1]>0$. By the assumption of monotonicity of $A$, the four derivatives appearing here are nonnegative, hence $\pa_u A_1>0$. Therefore by the implicit function theorem, $u_f$ belongs to the class $\mathscr{C}^1(\R_+^\ast)$, and for all $v>0$, $u_f'(v)=-\{\pa_v A_1/\pa_u A_1\}(u,u_f(v))$.


We then inject $u=u_f(v)$ in the second equation: we want to solve $a_2(u_f(v),v)v=g$. 
For $v>0$, we compute the derivative
\begin{align*}
\pa_v \{ A_2(u_f(v),v) \} = [u_f'(v) \pa_u A_2 + \pa_v A_2] (u_f(v),v)\\
=\frac{1}{\pa_u A_1} \, \det\begin{pmatrix}
\pa_u A_1 & \pa_v A_1 \\ \pa_u A_2 & \pa_v A_2
\end{pmatrix}(u_f(v),v)>0.
\end{align*}
The function $v\in\R_+\mapsto a_2(u_f(v),v)v\in\R_+$ is continuous, strictly increasing by the previous computation, and it is onto (by continuity and thanks to the lower bound for $a_2$). We write $v_{f,g}$ the only solution $v\ge0$ of $a_2(u_f(v),v)\, v=g$. Then $(u,v)=(u_f(v_{f,g}),v_{f,g})$ is the only solution of $A(u,v)=(f,g)$.
\end{proof}

\subsubsection{$A$ nonsingular on the closed set $\R_+^I$}
In the following, we prove the statement
\begin{Propo}\label{propo:Ahomeo}
Assume that the functions $a_i:\R_+^I\rightarrow\R_+$ are continuous and lower bounded by $\alpha>0$. Assume that on $\R_+^I$, $A$ is $\mathscr{C}^1$ and nonsingular with $\det \D(A)>0$. Then $A$ is a homeomorphism from $\R_+^I$ to itself and a $\mathscr{C}^1$ diffeomorphism from $(\R_+^*)^I$ to itself.
\end{Propo}

\begin{remark}
The restriction ``$A$ is $\mathscr{C}^1$ on the boundary'' does not allow the use of every power $x_i^{\gamma_{ij}}$ in the cross dependencies (they typically need to be bigger than $1$). However  
%
we can see in the proof (Step 3) that the assumption that $A$ is $\mathscr{C}^1$ and nonsingular on the closed set $\R_+^I$ is not optimal: it could be replaced by the weaker assumption that 
the restriction of $A$ on any half-(sub)space of the form $\prod_{i=1..I}\pi_i$, with $\pi_i=\{0\}$ or $\R_+^\ast$, is $\mathscr{C}^1$ and nonsingular.

This stronger version of the proposition would include the case of small (less than 1) power $x_i^{\gamma_{ij}}$.
\end{remark}
\begin{proof}

The proposition is essentially adapted from Hadamard global inverse mapping theorem and might be derived from \cite{Gordon72}. We prove here the main points. For convenience we write $A_|$ the restriction/corestriction of $A$ from $(\R_+^\ast)^I$ to itself, $A_|: (\R_+^\ast)^I\rightarrow (\R_+^\ast)^I, x\mapsto A(x)$. Note that indeed $A((\R_+^\ast)^I)\subset (\R_+^\ast)^I$ thanks to the positivity of the functions $a_i$, so that $A_|$ is well-defined.

\paragraph{Step 1. The functions $A$ and $A_|$ are proper.}
We claim that:

\vspace{2mm}

Inverse images by $A$ of compact set of $\R_+^I$ (resp. $(\R_+^*)^I$) are compact sets of $\R_+^I$ (resp. $(\R_+^*)^I$).

\vspace{2mm}

Let $K$ be a compact set of $\R_+^I$, then it is included in $[m,M]^I$ where $m>0$ if the compact is a subset of $(\R_+^*)^I$. Thanks to the positive lower bound $\alpha$ for the functions $a_i$, 
$$
A_i(x)\leq M \Rightarrow x_i\leq \alpha^{-1} M.
$$
It follows, by continuity of the $a_i$, that 
$$
\|x\|_\infty \leq \alpha^{-1} M \Rightarrow \max(a_i(x))\leq C(\alpha^{-1}M),
$$
and finally
$$
A(x)\in [m,M]^I \Rightarrow x\in [C(\alpha^{-1} M)^{-1}m, \alpha^{-1}M]^I=[m',M']^I.$$
We conclude using the continuity of $A$ that $A^{-1}(K)$ is then a closed bounded set of $[m',M']^I$ (with $m'>0$ for the case $(\R_+^*)^I $).

\paragraph{Step 2. The functions $A$ and $A_|$ are surjective.} We claim that $$A(\R_+^I)=\R_+^I \text{ and }  A((\R_+^*)^I)=(\R_+^*)^I.$$

We use the property that the image of an application which is proper and continuous is a closed set. Thanks to the previous step, this property applied to $A$ and $A_|$ gives that $A((\R_+^\ast)^I)$ is a closed set of $(\R_+^\ast)^I$ (for the induced topology on $(\R_+^\ast)^I$) and $A(\R_+^I)$ is a closed set (of $\R_+^I$).

The assumption of nonsingularity of $A$ implies by the implicit function theorem that $A((\R_+^\ast)^I)$ is also an open set of $(\R_+^\ast)^I$. By connectedness, we therefore have
\begin{equation*}
A((\R_+^\ast)^I)=(\R_+^\ast)^I.
\end{equation*}
Then since $A(\R_+^I)$ is a closed set of $\R_+^I$ containing $(\R_+^\ast)^I=A((\R_+^\ast)^I)$, we get the conclusion
\begin{equation*}
A(\R_+^I)=\R_+^I.
\end{equation*}


\paragraph{Step 3. The functions $A$ and $A_|$ are one-to-one.}

Finally, knowing that $A_|$ is onto from $(\R_+^*)^I$ to itself we can use theorem B in \cite{Gordon72}  to conclude that $A_|$ is a bijection.  To prove it is also the case for $A$ on $\R_+^I$, we only need to prove the injectivity on the boundary.

We write $\pa \R_+^I = \cup_i \{x\in \R_+^I : x_i=0\}$. Let us notice that thanks to the positivity of the functions $a_i$, 
 it suffices to show the injectivity on each of the spaces $\{x\in \R_+^I : x_i=0\}$. Without loss of generality,
 we consider the set $\{x\in \R_+^I : x_I=0\}$ and we want to show that on this set $A=(A_1,\cdots,A_{I-1},0)$ is one-to-one. Therefore, the initial problem of size $I$ reduces to the problem of size $I-1$ which consists in showing that the function $\tilde{A}:=(A_1,\cdots,A_{I-1})(\cdot,0):\R_+^{I-1}\rightarrow \R_+^{I-1}$ is one-to-one. The function $\tilde{A}$ is $\mathscr{C}^1$ and nonsingular since for all $\tilde{x}\in \R_+^{I-1}$,
 $$ 0< \det \D(A) (\tilde{x}, 0) = \det \begin{pmatrix}
 \D(\tilde{A})(\tilde{x}) & 0_{1,I-1} \\ \ast\cdots\ast & a_I(\tilde{x}, 0)
 \end{pmatrix} = a_I(\tilde{x}, 0)\, \det \D(\tilde{A})(\tilde{x}).
$$
Therefore, it satisfies the assumptions of Proposition \ref{propo:Ahomeo} with $I$ replaced by $I-1$. We conclude by iteration on the integer $I$, noticing that in the case $I=1$ we have $\pa \R_+^I = \{0\}$ and the injectivity on the unit set $\{0\}$ is obvious.

This ends the proof.

\end{proof}

\subsection{Elliptic estimates}

We start by recalling the following standard elliptic estimate (see for instance Theorem 2.3.3.6 in \cite{grisvard})

\begin{lem}\label{lem:ell} 
For any $p\in(1,\infty)$ and any regular open set $\Omega$, there exists positive constants $M_{p,\Omega}$ and $C_{p,\Omega}$ such that for all $M>M_{p,\Omega}$, 
\begin{align*}
\left.
    \begin{array}{ll}
& M w - \Delta w = f \in \L^p(\Omega), \\
&w\in \W^{2,1}_\nu(\Omega).
    \end{array}
\right \}\Longrightarrow \|w\|_{\W^{2,p}(\Omega)} \leq C_{p,\Omega} \|f\|_p \,.
\end{align*}
\end{lem}

Using this result we get the following useful Lemma:
\begin{lem}\label{Linfinityestimate}
 Let $f\in \L^\infty(\Omega)$, and let $w$ satisfy
$$w\in \H^2_\nu(\Omega),\; w\geq 0,\; \ -\Delta w\leq f \text{ in } \Omega.$$
Then there exists $C : = C(\Omega)$ such that
\begin{equation}\label{Linftylemma}
\|w\|_{\infty}\leq C\left(\|f\|_{\infty}+\| w\|_1\right). 
\end{equation}
\end{lem}
\begin{proof}
First, we fix $p\in (d/2,\infty)$ and $M>M_{p,\Omega}$ (see Lemma \ref{lem:ell}) and rewrite the equation as $M w-\Delta w\leq f+ Mw$. Using $w\geq 0$, the comparison principle, the elliptic estimate of Lemma \ref{lem:ell} and the Sobolev embedding $\W^{2,p}(\Omega) \hookrightarrow \L^\infty(\Omega)$, we get
\begin{align*}
\|w\|_{\infty}&\leq C  \left(\|f+M w\|_{p}\right)\leq C\left(\|f\|_{p}+M\|w\|_{p}\right)\\
&\leq C \left(\|f\|_{p}+ M \|w\|_{\infty}^{(p-1)/p}\|w\|_1^{1/p}\right)\\
&\leq C \left(\|f\|_{p}+\ep\|w\|_{\infty}+c(\ep)\|w\|_1\right) \mbox{ (Young's inequality)},
\end{align*}
and we conclude by choosing $\ep$ small enough.
\end{proof}
\begin{remark}{\rm Obviously, the conclusion of Lemma \ref{Linfinityestimate} would be the same when one only assumes that $f\in \L^p(\Omega)$, for some $p>d/2$.
}
\end{remark}

\medskip

{\bf{Acknowledgement}}: 
The research leading to this paper was funded by the french "ANR blanche" project Kibord: ANR-13-BS01-0004.
\medskip

\bibliographystyle{abbrv}
\bibliography{DLMT}

\end{document}